\documentclass[a4paper,12pt]{amsart} 
\usepackage{amsmath,amssymb,amsfonts, amsthm} 

\usepackage[usenames, dvipsnames]{xcolor}
\usepackage{tikz, tikz-cd}

\usepackage{booktabs, multirow}

\usepackage[unicode,bookmarks, pdftex]{hyperref}
\hypersetup{colorlinks=true,citecolor=NavyBlue,linkcolor=NavyBlue,urlcolor=Orange, pdfpagemode=UseNone, breaklinks=true}

\newcommand{\CC}{\mathbb{C}}
\newcommand{\FF}{\mathbb{F}}
\newcommand{\PP}{\mathbb{P}}
\newcommand{\QQ}{\mathbb{Q}}
\newcommand{\ZZ}{\mathbb{Z}}

\newcommand{\wt}{\widetilde}

\newcommand{\Ytilde}{\widetilde Y}

\newcommand{\Oh}{\mathcal{O}}

\newcommand{\gothM}{\mathfrak{M}}

\DeclareMathOperator{\Pf}{Pf}
\DeclareMathOperator{\Spec}{Spec}
\DeclareMathOperator{\Proj}{Proj}

\DeclareMathOperator{\Bl}{Bl}

\newtheorem{thm}{Theorem}[section]
\newtheorem{prop}[thm]{Proposition}
\newtheorem{cor}[thm]{Corollary}
\newtheorem{lem}[thm]{Lemma}
\newtheorem{rmk}[thm]{Remark}
\newtheorem{exam}[thm]{Example}

\newtheorem{ex}[thm]{Example}

\title{On T-divisors and intersections in $\overline \gothM_{1,3}$}

\author[S. Coughlan]{Stephen Coughlan}
\address{Stephen Coughlan\\Mathematisches Institut \\Lehrstuhl Mathematik VIII \\ Universit\"atsstra\ss e 30 \\ 95447 Bayreuth
\\Germany}
\email{stephen.coughlan@uni-bayreuth.de}

 \author[M. Franciosi]{Marco Franciosi}
 \address{Marco Franciosi\\Dipartimento di Matematica\\Universit\`a di Pisa \\Largo B. Pontecorvo 5\\I-56127  Pisa\\Italy}
 \email{marco.franciosi@unipi.it}
 
 \author[R.\ Pardini]{Rita Pardini}
 \address{Rita Pardini\\Dipartimento di Matematica\\Universit\`a di Pisa \\Largo B. Pontecorvo 5\\I-56127  Pisa\\Italy}
 \email{rita.pardini@unipi.it}

\author[J. Rana]{Julie Rana}
 \address{Julie Rana\\Department of Mathematics, Lawrence University, 711 E. Boldt Way, Appleton WI 54911, USA.}
\email{julie.f.rana@lawrence.edu}

 \author[S. Rollenske]{S\"onke Rollenske}
 \address{S\"onke Rollenske\\FB 12/Mathematik und Informatik\\
 Philipps-Univer\-si\-t\"at Marburg\\
 Hans-Meerwein-Str. 6\\
 35032 Marburg\\
 Germany}
 \email{rollenske@mathematik.uni-marburg.de}

\date{}

\begin{document}
\setcounter{tocdepth}{1}

\begin{abstract}
 The moduli space of  stable surfaces with $K_X^2 = 1$ and $\chi(X) = 3$ has at least two irreducible components that contain surfaces with T-singularities. We show that the two known components intersect transversally in a divisor. Moreover, we exhibit other new boundary divisors  and study how they intersect one another.
\end{abstract}

\maketitle

\tableofcontents

\section{Introduction}
Surfaces of general type and their moduli spaces is a classical subject of study that goes back to the beginning of the 20th century. 
The moduli spaces $\gothM_{K^2, \chi}$ classifying canonical models of surfaces of general type were constructed by Gieseker \cite{gieseker77} and we now have a modular compactification  $\overline\gothM_{K^2, \chi}$, the moduli space of stable surfaces,  due mainly to work of Koll\'ar, Shepherd-Barron, and Alexeev \cite{ksb88,alexeev94}. 

In this paper we continue the investigation of I-surfaces, that is, stable surfaces with $K_X^2 = 1$ and $\chi(X) = 3$, refining a conjectural picture developed by the last four authors in \cite{FPRR},
where I-surfaces with one T-singularity were studied. Note that T-singularities are precisely the quotient singularities that can occur in the closure of the Gieseker components of $\overline\gothM_{K^2, \chi}$.

Classical I-surfaces, i.e.~with canonical singularities,  or more generally I-surfaces with $K_X$ Cartier, are very well understood. They are double covers of the quadric cone in $\PP^3$
 branched over a quintic section and the vertex \cite{FPR17}.
 Nonetheless, progress on understanding their degenerations, or better all I-surfaces has been slow.

 A philosophically sound method consists of imposing the existence of one T-singularity  (i.e.~locally of the form $\frac{1}{dn^2}(1, dna-1)$). If the singularity imposes independent conditions on the stable surface then we get a stratum whose codimension coincides with the dimension of the local deformation space of the singularity.
 
For example, there is a T-divisor in 
$\overline\gothM_{1, 3}$ consisting of the stratum of I-surfaces with a singularity of type $\frac14 (1,1)$, obtained by allowing the branch locus of the double cover to pass through the vertex.
In \cite{FPRR} it was shown that the other I-surfaces with one T-singularity do not conform to this simple pattern.

 Our starting point  is the observation \cite[Cor.~1.2]{FPRR}  that  $\overline \gothM_{1,3}$
 has at least two irreducible components, both of dimension $28$:
  the Gieseker component $\gothM_{1,3}$ and the component $\gothM_{\text{RU}}$ whose general element is 
  an I-surface with a unique $\frac 1{25} (1, 14) $ singularity constructed by Rana and Urz\'ua in \cite{rana-urzua19}.

Confirming a suspicion from \cite{FPRR} we show that the two components intersect in a divisor parametrising  RU-surfaces of cuspidal type. We go on to investigate the intersections of the various T-divisors obtained so far.
  
  The schematic picture in Figure \ref{fig: schematic} illustrates our results, where we label each stratum by the singularities of its general element. 
  
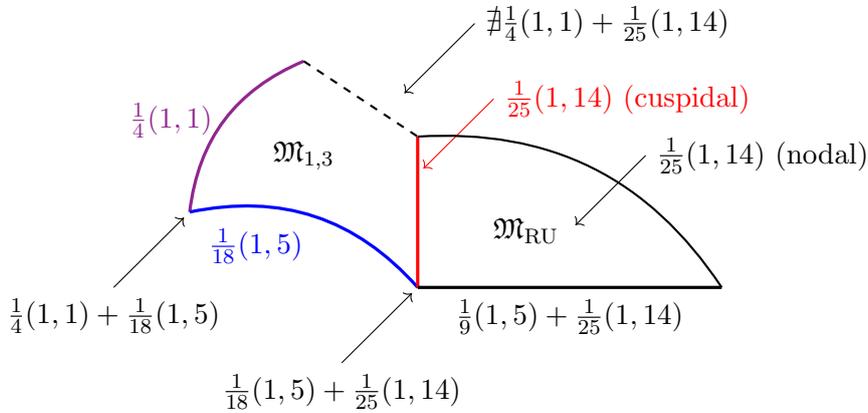
\begin{figure}[h]
\begin{center} \begin{tikzpicture}[thick, every node/.style = {font = \small}]
 \draw[Plum, very thick, bend right] (1.5,2) to node [left] { $\frac 14(1,1)$}  (0,0) ;
 \draw[blue, very thick, bend left] (0,0) to node [below left] { $\frac {1}{ 18}(1,5)$ } (3,-1);
  \draw[red, very thick] (3,1) to (3,-1);

\draw [thin, ->, shorten >=3, red]  (4,1.5) node[right]   {$\frac 1 {25}(1,14)$ (cuspidal)} -- (3,.5) ; 
  
  \node at (1.5, 0.75) {$\gothM_{1,3}$};
  
  \draw (3,1) to[bend left] (7,-1);
  
  \draw[very thick] (3,-1) to node[below ] {$\frac{1}9(1,5) + \frac{1}{25}(1,14)$} (7,-1);

  \draw[thin, ->, shorten >=3]  (6, .75) node[ right] {$\frac 1 {25}(1,14)$ (nodal)} to ++(-1, -1) node[left] {$\gothM_{\text{RU}}$} ;

   \draw[thin, ->, shorten >=3]  (-1, -1,0) node[below] {$\frac 14(1,1) + \frac {1}{ 18}(1,5)$ } -- ++(1,1);

    \draw[thin, ->, shorten >=3]  (2, -2,0) node[below] {$\frac {1}{ 18}(1,5)+\frac{1}{25}(1,14)$ } -- ++(1,1);

  \draw[dashed]  (1.5,2) to (3,1);
  \draw[thin, ->, shorten >=3]    (3.75, 2.5) node[right] {${\nexists}\frac{1}{4}(1,1) + \frac{1}{25}(1,14)$} to (2.75, 1.5);
   \end{tikzpicture}
   \caption{Schematic picture of (known parts of) the moduli space of I-surfaces }\label{fig: schematic}
\end{center}
\end{figure}

More precisely, we prove the following:
\begin{enumerate}
 \item I-surfaces with a $\frac1{25}(1, 14)$ singularity and of cuspidal type are smoothable; they form a divisor which is the intersection of the two components. (Section \ref{sec: smoothable})
 \item While the $\frac 19 (1,5)$ singularity does not occur in the closure of the Gieseker component, there is a divisor in the RU-component of surfaces with T-singularities $\frac 19 (1,5) + \frac{1}{25}(1, 14)$. (Section \ref{sec: divisor})
\item  The divisor of surfaces with an $\frac{1}{18}(1,5)$ singularity, known to be in the closure of the Gieseker component,  intersects the aforementioned divisors in a subset of codimension two parametrising surfaces with singularities $\frac {1}{ 18}(1,5)+\frac{1}{25}(1,14)$. (Prop.~\ref{prop: 3+5} and Ex.~\ref{ex: sing 5+3})
\item The divisors parametrising surfaces with a $\frac14(1,1)$ singularity respectively a $\frac{1}{18}(1, 5)$ singularity intersect, as expected, in a subset of codimension two, whose general element is a surface with singularities $\frac 14(1,1) + \frac {1}{ 18}(1,5)$. (Section \ref{subsection: sing 2+3})
\item We suspect the remaining intersection to be empty, but can only show that the combination $\frac{1}{4}(1,1) + \frac{1}{25}(1,14)$ cannot occur on an I-surface (Section \ref{sec: two sings}). If the two divisors intersect at all, then they do so in considerably more singular surfaces, and presumably in high codimension. 
\end{enumerate}
Two complementary approaches are used to establish these claims: geometric study of the minimal resolution and algebraic study of the canonical ring. Both points of view give us descriptions of the general surface in each stratum. Canonical rings are needed to establish deformations and (partial) smoothings via explicit equations, and the study of configurations of rational curves on the minimal resolution gives the non-existence result. 

Having established a complicated format for the canonical ring of the RU-surface in Corollary \ref{cor!index-5-format}, we include   a very surprising alternative description of the same surface as a hypersurface in weighted projective space in Section \ref{sect: RU hypersurface}.

\subsection*{Acknowledgements}

M.F.  and R.P. are  partially supported by the project PRIN 
 2017SSNZAW$\_$004 ``Moduli Theory and Birational Classification"  of Italian MIUR and members of GNSAGA of INDAM.  

J.R. and S.R. are grateful for the support of an NSF-AWM Mentoring Travel Grant, which allowed for an extended visit in June 2021. 

J.R. is partially supported by NSF LEAPS-MPS grant 2137577.

\section{Preliminaries}

We work over the complex numbers. Linear equivalence is denoted by $\sim$.

 For a Hirzebruch surface $\FF_n$,  we denote by $\sigma_{\infty}$ the negative section and by $\Gamma$ the class of a ruling, so that  a section $\sigma_0$ disjoint from $\sigma_{\infty}$ is linearly equivalent to
  $n\Gamma +\sigma_{\infty}$.

An I-surface $X$ is   a stable surface with $K_X^2=1$, $p_g=2$, $q=0$ 
(see   \cite{FPRR}).\footnote{  To exclude some more pathological (non-smoothable) examples (compare \cite{rollenske21}) we could be more specific and fix the Hilbert series of the canonical divisor  to be $h(t) = \frac{1-t^{10}}{(1-t)^2(1-t^2) (1-t^5)}$.}

A curve is an effective divisor, not necessarily irreducible or reduced.  For $n$  a positive integer, an irreducible  $(-n)$-curve $D$ on a smooth surface is a smooth rational curve with $D^2=-n$.
 
 A T-singularity $Q$ is either a rational double point or a   2-dimensional 
quotient singularity of type $\frac{1}{dn^2}(1, dna-1)$, where $n>1$ and  $d,a 
>0$ are integers with $a$ and $n$ coprime.
These are precisely the quotient singularities that admit a  $\QQ$-Gorenstein
smoothing, that is, that can occur on smoothable stable surfaces 
(cf.~\cite[\S3]{ksb88}).

The exceptional divisor of the minimal  resolution of  a T-singularity $\frac{1}{dn^2}(1, dna-1)$ is a so-called \emph{T-string}, a string of rational curves  with self-intersec\-tions
 $-b_1,-b_2, \ldots, -b_r$ given by the Hirzebruch--Jung continued fraction 
expansion $[b_1,b_2,\ldots, b_r]$ of $\frac{dn^2}{dna-1}$ (see, 
e.g.~\cite[Chapter 10]{CLS}). The index of $X$ at $Q$ is $n$.

\section{Computing the canonical ring of the RU-surface}\label{sec: canonical-ring}
For the reader's convenience we recall the construction of a general RU-surface. 
\begin{exam}\label{ex: RU}
Let $Y$ be an elliptic surface with $p_g(Y)=2$, $q(Y)=0$ such that:
\begin{itemize}
\item $Y$ has a $(-3)$-section  $A$
\item  all the elliptic   fibers are irreducible. 
\end{itemize}
By \cite[Lemma  3.8]{FPRR}, the surface $Y$ is a double cover $\pi\colon  Y\to \FF_6$ branched on a smooth divisor $D\in |\sigma_\infty+3\sigma_0|$.

Let $F_1$ be a singular fiber and let $q$ be its singular point: blow up $F_1$ at $q$ and then at a point $q_1$ infinitely near to $q$ and lying on the strict transform of 
$F_1$ to get a surface $\tilde Y$.
 The strict transform of $F_1$ is a $(-5)$-curve $B$, the strict transform of $A$ (that we still denote by $A$) is a $(-3)$-curve,  the strict transform of the curve of the first blow up is a $(-2)$-curve, which we call $C$, so that  $A,B,C$ is a string of type $[3,5, 2]$. Note that this is true both for  $F_1$ nodal and for $F_1$ cuspidal.  Then the string $A,B,C$ can be blown down to obtain an I-surface  with unique singularity of type $\frac{1}{25}(1,14)$, compare Figure \ref{fig: RU surface}. 
 
 Since we assumed the fiber $F_1$ to be irreducible it is either a nodal curve (type $I_1$) or a cuspidal curve (type $II$) and we call $X$ a nodal or cuspidal RU-surface respectively. 
 \end{exam}
We know from \cite[Prop.~3.13]{FPRR} that nodal RU-surfaces form an open subset of an irreducible component $\gothM_\text{RU}$ of the moduli space. In order to understand the interaction of this component with the Gieseker component, we study RU-surfaces from the point of view of canonical rings.

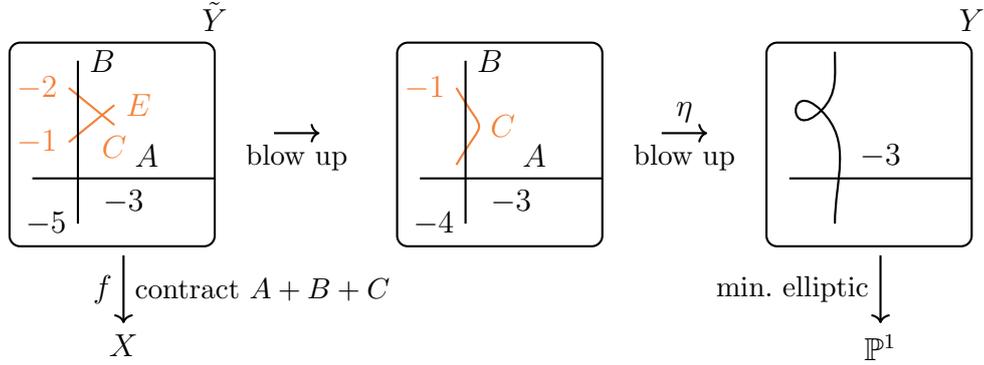
\begin{figure}
\begin{center}
\begin{tikzpicture}
[curve/.style ={thick, every loop/.style={looseness=10, min distance=30}},
exceptional/.style = {Orange},
scale = 0.6
]

\begin{scope}[xshift=-16.6cm]
\draw[thick, rounded corners] (-2.5, -1.5) rectangle (2, 3) node [above ] {$\tilde Y$};
\draw [curve, exceptional] (-1.2,.8) node[below, left]{$-1$} to  ++ ( 1, .82) node[ right]{$E$};
\draw [curve, exceptional] (-1.2,2) node[above,left ]{$-2$} to  ++(1,-.82) node[ below ]{$C$};
\draw [curve] (-1,-1) node[left]{$-5$} to  (-1,2.6) node[right] {$B$};
\draw [curve] (-2,0)  to node [below] {$-3$}  node[ above right] {$A$} (2,0);
\draw [curve,->] (3.3, 1) to node[above]  {} node[below] {\small blow up} ++(1,0);
\draw [curve,->] (0, -1.7) to node [left] { $f$} node [right] {\small contract $A+B+C$} ++(0,-1.5) node [below] {$X$} ;
\end{scope}

\begin{scope}[xshift=-8.1cm]
\draw[thick, rounded corners] (-2.5, -1.5) rectangle (2, 3) node [above right] {};
\draw [curve, exceptional] (-1.2,2) node[above,left ]{$-1$} to[bend left, looseness = 2]  node[  right]{$C$} ++(0,-1.7);
\draw [curve] (-1,-1) node[left]{$-4$} to  (-1,2.6) node[right] {$B$};
\draw [curve] (-2,0)  to node [below] {$-3$}  node[ above right] {$A$} (2,0);
\draw [curve,->] (3.3, 1) to node[above]  {$\eta$} node[below] {\small blow up} ++(1,0);
\end{scope}

\begin{scope}
\draw[thick, rounded corners] (-2.5, -1.5) rectangle (2, 3) node [above] {$Y$};
\draw [curve] (-1,-1) to[out=90, in=-45] (-1.3, 1.5) to[out=135, in =225,loop] () to[out=45, in = 270]  (-1,2.8);
\draw [curve] (-2,0)  to node [above] {$-3$}  (2,0);
\draw [curve,->] (0, -1.7) to  node[left]{\small min.\ elliptic} ++(0,-1.5) node [below] {$\PP^1$};
\end{scope}
\end{tikzpicture}
\end{center}
\caption{Construction of an RU-surface, nodal case}\label{fig: RU surface}
\end{figure}

\subsection{Strategy}
Let us explain the strategy underlying the algebraic computations along the following diagram.
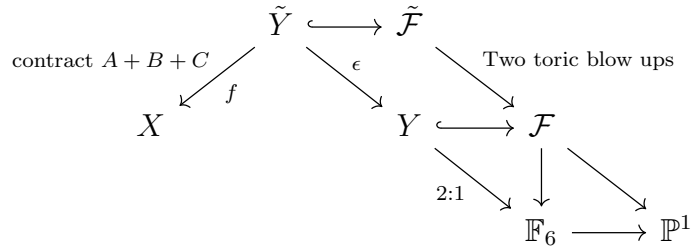
\[
\begin{tikzcd}
 {} &
  \tilde Y \arrow{dl}{f}[swap]{\text{contract $A+B+C$}} \arrow[hookrightarrow]{r}\arrow{dr}{\epsilon} &
 {\tilde {\mathcal F}}\arrow{dr}{\text{Two toric blow ups}}
 \\
 X& & Y  \arrow[hookrightarrow]{r}\arrow{dr}[swap]{2:1} & { \mathcal F}\dar\arrow{dr}\\
 &&& \FF_6\rar & \PP^1
 \end{tikzcd}
\]
\begin{enumerate}
\item Construct $Y$ in toric bundle $\mathcal{F}$ over $\PP^1$.
\item Blow up to $\tilde Y$ in a new toric variety $\tilde{\mathcal{F}}$.
\item Construct the canonical model of $X$ by writing $R(X,K_X)$ as a subring of the Cox ring of $\tilde Y$.
\item Study $\QQ$-Gorenstein deformations of $X$ by deforming the canonical ring. 
\end{enumerate}
The deformations in Step (4) were originally found by considering all deformations over the base $\Spec \CC[\varepsilon]/(\varepsilon^k)$ for $k=2$ and extending them to $k=3,4$ etc.~according to ideas of Reid \cite{reid90}. Rather than reproducing these unwieldy computations, we express the final results using formats for Gorenstein rings.

\subsection{The elliptic surface}
Instead of using the double cover $Y\to\FF_6$, which leads to an elliptic surface with fibers polarised in degree 2, we construct $Y$ via the halfpolarisation $A$, so that the fibers have degree 1.
Consider the toric variety $\mathcal{F}$ with Cox ring 
\[\begin{pmatrix}t_0&t_1&s_1&s_0&\zeta\\1&1&-3&0&0\\0&0&1&2&3\end{pmatrix}\]
and irrelevant ideal $I=(t_0,t_1)\cap(s_0,s_1,\zeta)$. Geometrically, this is a $\PP(1_{s_1},2_{s_0},3_{\zeta})$-bundle over $\PP^1_{t_0, t_1}$. Alternatively, $\mathcal{F}\cong(\CC^5\setminus V(I))/(\CC^*)^2$ where the action is determined by the columns of the above matrix.
The subvariety $\zeta=0$ is isomorphic to $\FF_6$.  Write $\Gamma$ for the fiber of $\FF_6\to\PP^1$ with base coordinates $t_0,t_1$. The fiber coordinates are $s_0,s_1^2$; $(s_0=0)$ being the positive section $\sigma_0$ and $(s_1^2=0)$ being the negative section $\sigma_{\infty}$. The elliptic surface $Y$ is a relative sextic in $\mathcal{F}$ and the halfpolarisation $A$ is a section cutting out one Weierstrass point on each fiber. 

\begin{lem}\label{lem: elliptic surface}
Let $Y$ be a double cover of $\FF_6$ branched over a smooth element $D$ of $|{-6\Gamma+4\sigma_0}|$. Then $Y$ is a hypersurface in $\mathcal{F}$ of bidegree 
$(0,6)$. 
 One can choose coordinates so that $Y$ has a nodal or cuspidal fiber $B$ over $(0,1)$ with singularity at the point with coordinates $(0,1;0,1,0)$. The equation of $Y$ can be written in the form:
\[(\zeta-\theta t_1^3s_0s_1)\zeta=s_0^3+t_0k'_{11}s_0s_1^4+t_0(t_0l'_{16}+\tau t_1^{17})s_1^6,\]
where $k'_{11}(t_0,t_1),l'_{16}(t_0,t_1)$ are general polynomials in $t_0,t_1$ of respective degrees $11,16$ and $\theta$, $\tau$ are parameters. 
In particular, the special fiber is nodal if $(\theta,\tau\ne0)$, cuspidal if $(\theta=0)$, type $I_2$ if $(\tau=0)$ and type III if $(\theta=\tau=0)$.\footnote{More precisely, in the latter two cases, the surface $Y$ is singular and its minimal resolution has a fiber of the given type.}
\end{lem}

\begin{proof}
The linear system $|{-6\Gamma+4\sigma_0}|$ on $\FF_6$ decomposes into $A+|3\sigma_0|$. By the above discussion, the branching over $A$ is inherited from the structure of the toric variety, and thus $Y$ has equation:
\[\zeta^2=s_0^3+j_{6}(t_0,t_1)s_0^2s_1^2+k_{12}(t_0,t_1)s_0s_1^4+l_{18}(t_0,t_1)s_1^6\]
where $\zeta$ is the double cover variable and $j,k,l$ are general polynomials  in $t_0,t_1$ of respective degrees $6,12,18$ so that the right hand side of the equation cuts out a general element of $|3\sigma_0|$. The coefficient of $s_0^3$ is non-zero because otherwise two components of $D$ would intersect giving a singularity. We normalise this coefficient to be $1$ and then we use the Tschirnhausen transformation $s_0\mapsto s_0+\frac13j_6 s_1^2$ to remove $j_6$.

The elliptic fibration $Y\to\PP^1$ has singular fibers when the discriminant $\Delta:=4k^3+27l^2$ vanishes. 
Assume that the fiber $B\colon(t_0=0)$ is singular, so that $\Delta|_B$ vanishes. Note that $\Delta|_B=4\alpha^3+27\beta^2$, where $\alpha$ (resp.~$\beta$) is the coefficient of $t_1^{12}$ in $k$ (resp.~$t_1^{18}$ in $l$). Hence there exists $\varepsilon$ with $\alpha=-3\varepsilon^2$ and $\beta=2\varepsilon^3$.

Next we use the coordinate change $s_0\mapsto s_0+\varepsilon t_1^6 s_1^2$ to move the singularity of $B$ onto the section $\sigma_0\colon (s_0=0)$, and the equation of $\tilde Y$ becomes
\[\zeta^2=s_0^3+3\varepsilon t_1^6s_0^2s_1^2+t_0k'_{11}(t_0,t_1)s_0s_1^4+t_0(t_0l'_{16}(t_0,t_1)+\tau t_1^{17})s_1^6.\]
If $\varepsilon$ is zero then the fiber $B$ is cuspidal. We write now $\theta^2:=3\varepsilon$ and
use the coordinates change $\zeta\mapsto\zeta+\theta t_1^3s_0s_1$ to move the $3\varepsilon t_1^6$-term to the left side. This gives the claimed equation.

Close to $(0,1;0,1,0)$ we can set $t_1 = s_1 =1$ so that the equation becomes
\begin{align*}\zeta^2 &=s_0^3+3\varepsilon s_0^2+t_0k'_{11}(t_0,1)s_0+t_0(t_0l'_{16}(t_0,1)+\tau )\\
 & = \tau t_0 + \left(3\epsilon s_0^2 + t_0s_0 k'_{11}(0,1) + t_0 ^2 l'_{16}(0,1)\right)+ \text{h.o.t.}, 
\end{align*}
 which defines a smooth surface if $\tau \neq 0$. Otherwise, if $\tau = 0$ (resp.~$\varepsilon=\tau=0$), we get a fiber of type $I_2$, respectively $I\!I\!I$ after resolving the $A_1$ surface singularity.
\end{proof}

Note that the choice of square-root $\theta$ corresponds to a choice of branch for the second blowup of the nodal fiber (cf.~Figure \ref{fig: RU surface}). 
The curious change of coordinates on $\zeta$ ensures that the blowups are at torus fixed points and thus in the next Section they can be expressed in terms of toric geometry.

\subsection{The resolution $\tilde Y$}
Consider now the toric variety $\tilde{\mathcal{F}}$ with Cox ring
\begin{equation}\label{eq: cox}
\begin{pmatrix}t_0&t_1&s_1&s_0&\zeta&c&e\\
1&1&-3&0&0&0&0\\
0&0&1&2&3&0&0\\
2&0&0&1&1&-1&0\\
1&0&0&0&1&1&-1
\end{pmatrix}
\end{equation}
and irrelevant ideal 
\begin{multline*}
(t_0,t_1)\cap(s_0,s_1,\zeta)\cap(c,t_1)\cap(c,s_1)\cap(t_0,s_0,\zeta)\\
\cap(e,t_1)\cap(e,s_0)\cap(e,s_1)\cap(t_0,\zeta,c). 
\end{multline*}
 This is the toric blow up of $\mathcal{F}$ at the subscheme $t_0^2=s_0=\zeta=0$ with exceptional divisor $(c=0)\cong\PP(1,1,2)$, followed by the toric blowup at the subscheme $t_0=\zeta=c=0$ with exceptional divisor $(e=0)\cong\PP^2$. Indeed, consider the subvariety $(e=0)$. The structure of the irrelevant ideal implies that $t_1,s_0,s_1$ are non-zero. Thus we apply part of the $(\CC^*)^4$-action (first three rows of the matrix \eqref{eq: cox}) to rescale these three coordinates $t_1=s_0=s_1=1$. We are left with the $\CC^*$-quotient of $\CC^3_{t_0,\zeta,c}$ induced by the last row with irrelevant ideal $(t_0,\zeta,c)$. In the same way, $(c=0)$ is the toric subvariety with Cox ring
\[\begin{pmatrix}
t_0&s_0&\zeta&e\\
2&1&1&0\\
1&0&1&-1
\end{pmatrix}\]
and irrelevant ideal $(s_0,e)\cap(t_0,\zeta)$. That is, the blowup of $\PP(2_{t_0},1_{s_0},1_{\zeta})$ at the point $(0,1,0)$.

\begin{rmk} A reference for this approach to toric varieties is Chapter 14 of the book \cite{CLS}. The generators of the Cox ring (columns of \eqref{eq: cox}) correspond to primitive generators of the rays in the fan $\Sigma$ of $\tilde{\mathcal{F}}$ via Gale duality, and the irrelevant ideal encodes the cones of $\Sigma$. For example, the generator $c$ gives a ray inside the cone $\sigma$ spanned by primitive vectors $v_{t_0}$, $v_{s_0}$, $v_{\zeta}$ with primitive generator $v_c=2v_{t_0}+v_{s_0}+v_{\zeta}$ determined by the third row. The components of the irrelevant ideal involving $c$ ensure that the cones of $\Sigma$ include the barycentric subdivision of $\sigma$, with respect to the ray generated by $v_c$.

\end{rmk}

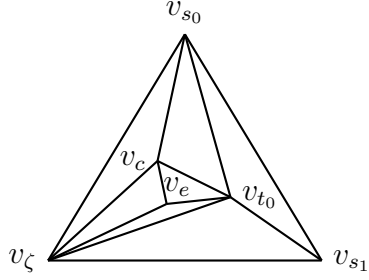
\begin{figure}
\begin{center}
\begin{tikzpicture}
[curve/.style ={thick, every loop/.style={looseness=10, min distance=30}},
exceptional/.style = {Orange},
scale = 0.6
]
\begin{scope}[xshift=-8.5cm]
\draw [curve] (3,0) to node[above,pos=1]{$v_{s_0}$}  (0,5);
\draw [curve] (0,5) to node[above,left,pos=1]{$v_{\zeta}$} (-3,0);
\draw [curve] (-3,0)  to node[above,right,pos=1]{$v_{s_1}$} (3,0);
\draw [curve] (-3,0)  to node[above,right,pos=1]{$v_{t_0}$} (1,1.4);
\draw [curve] (3,0)  to node{} (1,1.4);
\draw [curve] (0,5)  to node{} (1,1.4);

\draw [curve] (-3,0)  to node[above,left,pos=1]{$v_c$} (-0.6,2.2);
\draw [curve] (0,5)  to node{} (-0.6,2.2);
\draw [curve] (1,1.4)  to node{} (-0.6,2.2);

\draw [curve] (-3,0)  to node[above,left,pos=1.3]{$v_e$} (-0.4,1.25);
\draw [curve] (-0.6,2.2)  to node{} (-0.4,1.25);
\draw [curve] (1,1.4)  to node{} (-0.4,1.25);

\end{scope}
\end{tikzpicture}
\end{center}
\caption{A cross section of the fan $\Sigma(\tilde{\mathcal{F}})$, with the origin and $v_{t_1}$ behind the page.}
\end{figure}

\begin{lem}\label{lem!Ytilde} The double blowup $\tilde Y$ of $Y$ is a hypersurface in $\tilde{\mathcal{F}}$ of multidegree $(0,6,2,1)$.
The equation of $\tilde Y$ can be written in the form:
\begin{equation}\label{eqn!Ytilde}
(e\zeta-\theta t_1^3s_0s_1)\zeta=cs_0^3+cet_0k'_{11}s_0s_1^4+t_0(c^2e^3t_0l'_{16}+\tau t_1^{17})s_1^6
\end{equation}
where $k'_{11}=k'_{11}(c^2e^3t_0,t_1)$, $l'_{16}=l'_{16}(c^2e^3t_0,t_1)$ 
\end{lem}
\begin{proof}
The first blow up is at the subscheme $t_0^2=s_0=\zeta=0$ which is supported at the singularity of $B$. We add one new variable $c$ to the Cox ring of $\mathcal{F}$ and read off the blowup map using the third row of the Cox grading \eqref{eq: cox}. We use the same labels for the new coordinates:
\[t_0\mapsto c^2t_0,\ t_1\mapsto t_1,\ s_0\mapsto cs_0,\ s_1\mapsto s_1,\ \zeta\mapsto c\zeta
\]
Note the weighting on $t_0$. The equation of the first blowup $Y^{(1)}$ is obtained by pulling back the equation of $Y$ under this map and dividing by $c^2$:
\[(\zeta-\theta t_1^6s_0s_1)\zeta=cs_0^3+ct_0k'_{11}
(c^2t_0,t_1)s_0s_1^4+t_0(c^2t_0l'_{16}(c^2t_0,t_1)+\tau t_1^{17})s_1^6.\]

The second blowup is at $\zeta=c=t_0=0$ and the map is read off from the fourth row of the Cox grading \eqref{eq: cox}:
\[t_0\mapsto et_0,\ t_1\mapsto t_1,\ s_0\mapsto s_0,\ s_1\mapsto s_1,\ \zeta\mapsto e\zeta,\ c\mapsto ec
\]
The equation of the second blowup $\tilde Y=Y^{(2)}$ is then as claimed in \eqref{eqn!Ytilde}.
\end{proof}

\begin{rmk}
According to the discussion before Lemma \ref{lem!Ytilde}, the exceptional curve $C=\tilde Y\cap (c=0)$ is obtained by substituting $c=0$ and $t_1=s_1=1$ into equation \eqref{eqn!Ytilde}:
\[C\colon \left( (e\zeta-\theta s_0)\zeta=\tau t_0\right)\subset\Bl_{(0,1,0)}\PP(2_{t_0},1_{s_0},1_{\zeta}).\]
Recall that the parameter $\tau$ is the coefficient of $t_1^{17}$ in $l'_{17}$.

Thus $C$ is usually an irreducible rational curve. If $\theta=0$, then nothing especially interesting happens to $C$.
On the other hand, if $\tau=0$, then $C$ breaks into two rational curves meeting in a singularity on $\tilde Y$. In this situation, $Y$ has an $A_1$-singularity at the node of the fiber $B$. This case is treated in more detail later.
If $\theta=\tau=0$ then $C$ is a nonreduced double curve and $Y$ has an $A_1$ singularity at the node of $B$.

In a similar way, we find that the second exceptional curve is $E\colon (e=0)$, which implies $t_1=s_0=s_1=1$ with equation
\[E\colon\left(\theta\zeta=c+\tau t_0\right)\subset\PP^2_{t_0,\zeta,c}.\]
\end{rmk}

\subsection{Canonical rings from Cox rings}
We denote the Cox ring of $\tilde{\mathcal{F}}$ by $S$. This is a $\ZZ^4$-graded polynomial ring with generators $t_0$, $t_1$, $s_0$, $s_1$, $\zeta$, $c$, $e$. There is a $\ZZ$-linear coordinate change on $A_1(\tilde{\mathcal F})$, i.e., applied to the degrees, which shifts the $\mathbb{Z}^4$-grading on $S$ to:
\begin{equation}\label{eq: weight matrix}
\begin{pmatrix}
s_1 & t_0 & c & e & t_1 & s_0 & \zeta \\
1 & &&&0 & 2 & 3\\
&1&&&1 & 6 & 9 \\
&&1&&2 & 11 & 17 \\
&&&1&3 & 17 & 25
\end{pmatrix}
\end{equation}
By Lemma \ref{lem!Ytilde}, for $d\in\ZZ^4$ we have maps
\[S_d\to H^0(\Ytilde,d_1A+d_2B+d_3C+d_4E)\]
where the divisor classes are 
\[A\colon(s_1=0) ,\  B\colon(t_0=0) ,\  C\colon(c=0) ,\  E\colon(e=0),\   \Gamma\colon(t_1=0)\] 
on $\tilde Y$. Note that $(s_1=0)$ cuts out $A$ because $A$ is pulled back from the irreducible component $(s_1^2=0)$ of the branch locus of the double cover $Y\to\FF_6$. Thus the grading records the various linear equivalences on $\tilde Y$: 
\begin{align*}
 B&\sim \Gamma-2C-3E , \\ (s_0=0)&\sim 2A+6B+11C+17E ,\\  (\zeta=0)&\sim 3A+9B+17C+25E.
\end{align*}

The pullback of $K_X$ to $\tilde Y$ is 
\begin{align*}
{\Gamma+C+2E}+\tfrac15(3A+4B+2C)& \sim {B+3C+5E}+\tfrac15(3A+4B+2C)
\\
& \sim 5 E + \tfrac1 5 ( 3A + 9 B + 17 C). 
\end{align*}

\subsection{Generators}\label{sec!gens}
We want to construct the canonical ring $R(X,K_X)$ as the image of a map from the Cox ring $S$. 
To avoid having to keep track of corrections to multiplication maps in $R(X,K_X)$ due to rounding, we introduce formal $5$-th roots of $s_1,t_0,c$: $\alpha^5=s_1$, $\beta^5=t_0$, $\gamma^5=c$. 
The extension $S[\alpha,\beta,\gamma]$ is then a $(\frac15\ZZ)^3\oplus\ZZ$-graded ring containing $S$. In what follows, we consider  the ring homomorphism 
\[\iota \colon \bigoplus_{n\ge0} S[\alpha,\beta,\gamma]_{(\frac35n,\frac95n,\frac{17}5n,5n)}\to\bigoplus_{n\ge0} H^0(\tilde Y,nf^*K_X)\cong R(X,K_X)\]
The proof that this ring homomorphism is in fact surjective, is Corollary \ref{cor!is-canonical-ring}.

Let  $R'$ be the image of  $\iota$ and $X'=\Proj (R')$.

\begin{lem}\label{lem!generators} The graded ring 
\[R' \subseteq \bigoplus_{n\ge0}H^0\left(\tilde Y,n(5E+\tfrac15(3A+9B+17C)\right)\]
is generated by
\begin{align*}
x_0&=\alpha^3\beta^9\gamma^{17}e^5 &\text{deg 1}\\
x_1&=\alpha^3\beta^4\gamma^{7}e^2t_1&\text{deg 1}\\
y&=\alpha^6\beta^3\gamma^4et_1^3&\text{deg 2}\\
w&=\alpha^9\beta^2\gamma t_1^5&\text{deg 3}\\
u_0&=\alpha^2\beta^6\gamma^{13}e^3 s_0 & \text{deg 4}\\
u_1&=\alpha^2\beta\gamma^3 t_1s_0& \text{deg 4}\\
z&=\zeta& \text{deg 5}\\
t&=\alpha\beta^3\gamma^9e s_0^2& \text{deg 7}\\
g&=\alpha\beta^3\gamma^{14}s_0^5& \text{deg 17}\\
\end{align*}
\end{lem}
 
\begin{proof}The generators can be determined algorithmically as follows.
The grading on $S[\alpha, \beta, \gamma]$ is induced by the map $\delta\colon\ZZ^7\to(\frac15\ZZ)^3\oplus\ZZ$,
\[\delta(n_1,\dots,n_7)= 
\left(\tfrac{n_1}5, \tfrac{n_2}5,  \tfrac{n_3}5,  n_4,  n_5,  n_6,  n_7\right)
\begin{pmatrix}
1 & &&&0 & 2 & 3\\
&1&&&1 & 6 & 9 \\
&&1&&2 & 11 & 17 \\
&&&1&3 & 17 & 25
\end{pmatrix}^t
\]
via $\deg(\alpha^{n_1}\beta^{n_2}\gamma^{n_3}e^{n_4}t_1^{n_5}s_0^{n_6}\zeta^{n_7})=\delta(n_1,\dots,n_7)$.
Let $\Sigma=\ZZ^+\cdot(\frac35,\frac 95,\frac{17}5,5)$ be the cone in $(\frac15\ZZ)^3\oplus\ZZ$ generated by $\frac15(3A+9B+17C)+5E$, the pullback of $K_X$. 
Then the intersection of the preimage $\delta^{-1}\Sigma$ with the positive octant $(\frac15\ZZ^+)^3\oplus\ZZ^+$, is the cone of monomials in $R'$.
The primitive generators of this cone can be found using standard Hilbert basis algorithms for lattice cones (see e.g.~Chapter 7 of \cite{MS})  and these are the generators of $R'$.
\end{proof}

\begin{rmk}
It follows from Lemma \ref{lem!rolling} below, that the generator $g$ in degree 17 is eliminated by relations.
\end{rmk}

\begin{rmk}\label{rem: fixed part prelim}
At this stage, we already note that the fixed part of the canonical linear system is the image of the curve $E$, because both $x_0=\alpha^3\beta^9\gamma^{17}e^5$ and $x_1=\alpha^3\beta^4\gamma^{7}e^2t_1$ are divisible by powers of $e$. The other common factors correspond to curves $A,B,C$ which are contracted to the $\frac1{25}(1,14)$-point. A more precise description of the fixed part can be found in Remark \ref{rem: fixed part}.
\end{rmk}

\subsection{Relations}\label{sec!relations}
There are 10 binomial relations between the generators found in Lemma \ref{lem!generators} (excluding those involving $g$). These define a cone over the degree 5 generator $z$ in $\PP(1,1,2,3,4,4,5,7)$:
\begin{align*}
R_1\colon x_0y-x_1^3&=0 & R_2\colon x_0w-x_1^2y&=0 \\
R_3\colon x_1w-y^2&=0 & R_4\colon x_0u_1-x_1u_0&=0\\
R_5\colon x_1^2u_1-yu_0&=0 & R_6\colon x_1yu_1-wu_0&=0 \\
R_7\colon x_0t-u_0^2 &=0 & R_8\colon x_1t-u_0u_1&=0\\
R_9\colon  x_1u_1^2-yt&=0 & R_{10}\colon yu_1^2-wt&=0
\end{align*}

The equation \eqref{eqn!Ytilde} defining $\tilde Y$ induces several relations which cut out $X'$ inside this cone:
\begin{lem}\label{lem!rolling} There are four relations induced by \eqref{eqn!Ytilde} and these can be written as:
\begin{align*}
R_{11}\colon &&x_0P+x_1^2(\theta u_1z+\tau w^3)+u_0t &=0\\
R_{12}\colon &&x_1P+y(\theta u_1z+\tau w^3)+u_1t&=0\\
R_{13}\colon &&yP+w(\theta u_1z+\tau w^3)+u_1^3&=0\\
R_{14}\colon &&u_0P+x_1u_1(\theta u_1z+\tau w^3)+t^2&=0
\end{align*}
where $P_{10}=z^2+\dots$ is a general homogeneous form of degree 10 and $\theta$, $\tau$ are parameters.
The generator $g$ of degree $17$ is eliminated by a fifth relation $tP=\dots+g$.
\end{lem}
\begin{proof}
The leading term of the equation \eqref{eqn!Ytilde} cutting out $\tilde Y$ in $\tilde{\mathcal{F}}$ is $e\zeta^2$. By Lemma \ref{lem!generators} the unique generator involving $\zeta$ is $z$, and $e$ appears in $x_0$, $x_1$, $y$, $u_0$, $t$, thus we expect five relations induced by \eqref{eqn!Ytilde}.

We first study the relation $R_{11}$ more precisely. Using Lemma \ref{lem!generators} we can write in  $S[\alpha,\beta,\gamma]$
\[x_0z^2 = \left(\alpha^3\beta^9\gamma^{17}e^5\right) \zeta^2= \left(\alpha^3\beta^9\gamma^{17}e^4\right)\left(e\zeta^2\right),\]
and multiplying \eqref{eqn!Ytilde} with the excess monomial $\alpha^3\beta^9\gamma^{17}e^4$, we get
for the left hand side
\begin{align*}
 \alpha^3\beta^9\gamma^{17}e^4\left(e\zeta-\theta t_1^3s_0s_1\right)\zeta
 & = x_0z^2 - \theta \left(\alpha^3\beta^4\gamma^{7}e^2t_1\right)^2\left(\alpha^2\beta\gamma^3 t_1s_0\right)\zeta\\
  &= x_0z^2 -\theta x_1^2 u_1 z
\end{align*}
and for the right hand side
\begin{align*}
& \phantom{=}\;\; \alpha^3\beta^9\gamma^{17}e^4\left(cs_0^3+cet_0k'_{11}s_0s_1^4+t_0(c^2e^3t_0l'_{16}+\tau t_1^{17})s_1^6\right) \\
&= \left(\alpha^2\beta^6\gamma^{13}e^3s_0\right)\left(\alpha\beta^3\gamma^9es_0^2\right) +\tau \left(\alpha^3\beta^4\gamma^7e^2t_1\right)^2\left(\alpha^9\beta^2\gamma t_1^5\right)^3 \\
&\phantom{=}\;\;+  \left(\alpha^3\beta^9\gamma^{17}e^5\right)
\left( s_1^2\left(ct_0k'_{11}s_0s_1^2+t_0^2c^{2}e^2l'_{16}s_1^4\right)\right)\\
& = u_0t + \tau x_1^2 w^3 + x_0 \tilde P\end{align*}
Setting the two expressions equal and rearranging the terms we get a relation of the form 
\[x_0P+x_1^2(\theta u_1z+\tau w^3)+u_0t =0.\]
Note that the monomial $u_0t$ corresponds to the term $cs_0^3$ from \eqref{eqn!Ytilde}, and most of the terms (including $x_0z^2$ itself) are wrapped up in the general element $P$ of degree 10.

We can repeat this procedure to get the remaining relations, but it is easier to derive them from $R_{11}$. Using $R_1$ and $R_4$ we have: $x_0:x_1=x_1^2:y=u_0:u_1$. Thus writing $R_{12}=\tfrac{x_1}{x_0}R_{11}$ and applying these ratios gives the claimed relation.

The relation $tP=\dots $ always eliminates $g$, because $g$ corresponds to the term $cs_0^3$ in \eqref{eqn!Ytilde}, which always appears with nonzero coefficient.
\end{proof}

\begin{rmk}\label{rmk!index-1-cover} Recall the following characterisation of $T$-singularities via their index 1 canonical covers:
\[\frac1{dn^2}(1,dna-1)\cong(xy-z^{dn})\subset\frac1n(1_x,-1_y,a_z),\]
see e.g.~\cite[Prop.~3.10]{ksb88} or \cite[Ex.~2.1.8]{hacking-barcelona}.
\end{rmk}

\begin{cor}\label{cor!index-5-format} The 14 relations defining $X' \subset\PP(1,1,2,3,4,4,5,7)$ fit into a Pfaffian presentation as follows:
\[\Pf_4 M=MV=0\]
where
\[
M=\begin{pmatrix}
0&0&x_0 & x_1^2 & u_0 \\
&0&x_1 & y & u_1 \\
&&u_0 & x_1u_1 & t\\
&&&t&-(\tau w^3+\theta u_1z)\\
-\text{sym}&&&&P
\end{pmatrix},\
V=\begin{pmatrix}
0\\-u_1^2\\ 0 \\w\\-y\\0\\
\end{pmatrix}, 
\]
and we omit the diagonal zero entries of the $6\times 6$ skew matrix $M$. 
For general choices of $P_{10}$ and  $\tau$ the singular locus of $X'$ is one point of type $\frac1{25}(1,14)$. 
\end{cor}
\begin{proof}Write out the matrix product and the Pfaffians and compare this with the list of relations.

For the singular locus, $X'$ has a covering by orbifold affine charts $U_m\colon(m\ne0)$ centred at each coordinate point $P_m$ of the ambient space. Despite the number of equations, it is straightforward to show the nonsingularity (except $U_z$) of each chart. Here are a couple of sample computations. The chart $U_y$ is contained in $U_{x_1}\cap U_{w}$ because $R_3|_{y=1}$ gives $x_1w=1$, thus we ignore $U_y$. Infact, similar considerations show that $U_{x_1}=U_y\subset U_{x_0}\cap U_w$ and $U_{u_0}\subset U_{x_0}\cap U_t$. Hence we only need to check nonsingularity of four charts $U_{x_0}$, $U_w$, $U_{u_1}$ and $U_t$. For $U_{x_0}$, we use $R_1,R_2,R_4,R_7$ to eliminate $y,w,u_1,t$ respectively, giving a hypersurface in $\CC^4_{x_1,u_0,z}$ induced by any one of $R_{11},\dots,R_{14}$. It is then easy to show that this chart is nonsingular. The other three charts work in a similar way.

Since we know that the other charts of $X'$ are nonsingular, we only need to consider the orbifold chart $U_z$ in an analytic neighbourhood of $P_z\in X$. 
We use $R_{11},R_{12},R_{13},R_{14}$ to eliminate $x_0,x_1,y,u_0$ respectively. Thus the local coordinates near $P_z$ are $w,u_1,t$ and $R_{10}$ defines $X'$ locally 
as the hypersurface $u_1^5-wt+\text{h.o.t}=0$ in $\frac15(3_w,4_{u_1},2_t)\overset{\cdot2}{\cong}\frac15(1,3,4)$ 
after substituting $y=u_1^3+\text{h.o.t.}$ using $R_{13}$. By Remark \ref{rmk!index-1-cover}, this is a $\frac1{25}(1,14)$ singularity. 
\end{proof}

\begin{cor}\label{cor!is-canonical-ring}
The coordinate ring $R'$ of $X'\subset\PP(1,1,2,3,4,4,5,7)$ described in Corollary \ref{cor!index-5-format} is the canonical ring $R(X,K_X)$ and $X'\cong X$  is an RU-surface of nodal type if $\theta \neq 0$ and of cuspidal type if $\theta =0$.
\end{cor}

\begin{proof}A standard computer calculation shows that the minimal free resolution of $\Oh_{X'}$ as an $\Oh_\PP$-module is
\[0\to\Oh(-28)\to\Oh(-28)\otimes\mathcal{L}_1^\vee\to\Oh(-28)\otimes\mathcal{L}_2^\vee\to\mathcal{L}_2\to\mathcal{L}_1\to\Oh\to\Oh_{X'},\]
where
$\mathcal{L}_1=\bigoplus_{d\in L_1}\Oh(-d)$ with \[L_1=(3,4^2,5,6,7,8^2,9,10,11^2,12,14)
 \]
  and 
$\mathcal{L}_2=\bigoplus_{d\in L_2}\Oh(-d)$ with 
\[ L_2=(5,6,7^2,8^3,9^2,10^3,11^4,12^3,13^4,14^3,15^3,16^2,17,18^2,19).\]

Thus the coordinate ring $R'$ is Cohen--Macaulay by the Aus\-lander--Buchsbaum formula \cite[\S1.3]{BH}. Combining this with regularity in codimension 1, which follows from Corollary \ref{cor!index-5-format}, we see that $X'$ is projectively normal by Serre's criterion for normality. The dualising sheaf of $X'$ is 
$\omega_{X'}=\mathcal{E}xt^5_{\Oh}(\Oh_{X'},\omega_{\PP})=\Oh_{X'}
(28-1-1-2-3-4-4-5-7)=\Oh_{X'}(1)$ where 
$\mathcal{E}xt^5_{\Oh}(\Oh_{X'},\Oh)=\Oh(28)$ can be read off from the 
resolution above \cite[\S3.6]{BH}.  By projective normality, we obtain 
$H^0(X,nK_X)\cong H^0(X',\Oh_{X'}(n))$.
Hence we have $X'\cong X$,   $p_g(X)=h^0(\Oh_{X'}(1))=2$, and $K_{X}^2=1$ follows from $P_2(X)=h^0(\Oh_{X'}(2))=4$ and the Riemann--Roch formula of Blache \cite{bla95a}.
\end{proof}
\begin{rmk} The vanishing of $\theta$ corresponds to a cuspidal singular fiber as explained above. The vanishing of $\tau$ imposes an extra $\frac1{9}(1,5)$ point on $X'$ (these are two independent conditions in moduli).
\end{rmk}

\begin{rmk}\label{rem: fixed part} As promised in Remark \ref{rem: fixed part prelim}, we describe the fixed part $F$ of $|K_X|$. Indeed, $e$ divides all generators except $w,u_1,z$. Thus the image of $E$ is obtained by restricting the relation $R_{13}$ to the locus $x_0=x
_1=y=u_0=t=0$:
\[F\colon \left(w(\theta u_1z+\tau w^3)+u_1^3=0\right)\subset\PP(3_w,4_{u_1},5_z).\]
For general $\theta,\tau$, the curve $F$ passes through the index 5 point and has a node there.
If $\theta=0$ then $F$ is a cone with vertex $P_z$. If $\tau=0$ then $F$ has two components each passing through $P_w$ and $P_z$. If $\theta=\tau=0$ then $F$ is the triple line joining $P_w$ and $P_z$.

\end{rmk}

\subsection{A hypersurface in weighted projective space}\label{sect: RU hypersurface}

We give an alternative description of the general RU-surface, which is algebraically much simpler. The following result is inspired by the observation that $\tilde {\mathcal F}$ is birational to $\PP(1,3,17,25)$, as can be read of from the last row of the weight matrix \eqref{eq: weight matrix}. The decomposition of this map into extremal contractions is briefly described in Figure \ref{fig: toric-map}, where the red loci denote exceptional locus of each factor (for more details on how to compute this, see \cite[Chapter 15]{CLS}). Indeed, modulo flips, the birational map is a composition of three divisorial contractions $D_{s_1}$, $D_c$ and $D_{t_0}$ whose restrictions to $\tilde Y$ contract the T-chain $A+B+C$.

\begin{figure}
\begin{center}
\begin{tikzpicture}
[curve/.style ={thick, every loop/.style={looseness=10, min distance=30}},
exceptional/.style = {Orange},
scale = 0.3
]
\begin{scope}[xshift=-24.5cm]
\draw [curve] (3,0) to node[above,pos=0.95]{$\scriptstyle{s_0}$}  (0,5);
\draw [curve] (0,5) to node[above,left,pos=0.95]{$\scriptstyle{\zeta}$} (-3,0);
\draw [curve] (-3,0)  to node[above,right,pos=0.95]{$\scriptstyle{s_1}$} (3,0);
\draw [curve] (-3,0)  to node[above,right,pos=1]{} (1,1.25);
\draw [curve] (3,0)  to node{} (1,1.25);
\draw [curve] (0,5)  to node{} (1,1.25);

\draw [curve,color=red] (-3,0)  to node[above,left,pos=1]{} (-0,2.2);
\draw [curve] (0,5)  to node{} (-0,2.2);
\draw [curve] (1,1.25)  to node{} (-0,2.2);

\draw [curve] (-3,0)  to node[above,right,pos=1]{} (-0.7,1.25);
\draw [curve] (-0,2.2)  to node{} (-0.7,1.25);
\draw [curve] (1,1.25)  to node{} (-0.7,1.25);
\draw [curve,dashed,->] (3,3) to node[above]  {\tiny flip} node[below] {\small} (5,3);
\end{scope}

\begin{scope}[xshift=-16.5cm]
\draw [curve] (3,0) to node[above,pos=1]{}  (0,5);
\draw [curve] (0,5) to node[above,left,pos=1]{} (-3,0);
\draw [curve] (-3,0)  to node[above,right,pos=1]{} (1,1.25);
\draw [curve] (3,0)  to node{} (1,1.25);
\draw [curve] (0,5)  to node{} (1,1.25);
\draw [curve] (-3,0)  to node[color=red,pos=1]{$\bullet$} (3,0);

\draw [curve] (0,5)  to node[above,left,pos=1]{} (-0,2.2);
\draw [curve] (0,5)  to node{} (-0.7,1.25);
\draw [curve] (1,1.25)  to node{} (-0,2.2);

\draw [curve] (-3,0)  to node[above,left,pos=1.3]{} (-0.7,1.25);
\draw [curve] (-0,2.2)  to node{} (-0.7,1.25);
\draw [curve] (1,1.25)  to node{} (-0.7,1.25);
\draw [curve,->] (3,3) to node[above]  {\tiny contr.} node[below] {$\scriptstyle{D_{s_1}}$} (5,3);
\end{scope}

\begin{scope}[xshift=-9cm]
\draw [curve] (0,5) to node[above,left,pos=1]{} (-3,0);
\draw [curve,color=red] (-3,0)  to node[color=black,above,right,pos=0.95]{$\scriptstyle{t_0}$} (1,1.25);
\draw [curve] (0,5)  to node{} (1,1.25);

\draw [curve] (0,5)  to node[above,left,pos=1]{} (-0,2.2);
\draw [curve] (0,5)  to node{} (-0.7,1.25);
\draw [curve] (1,1.25)  to node{} (-0,2.2);

\draw [curve] (-3,0)  to node[above,left,pos=1.3]{} (-0.7,1.25);
\draw [curve] (-0,2.2)  to node{} (-0.7,1.25);
\draw [curve] (1,1.25)  to node{} (-0.7,1.25);
\draw [curve,dashed,->] (1.5,3) to node[above]  {\tiny flip} node[below] {\small } (3,3);
\end{scope}

\begin{scope}[xshift=-2cm]
\draw [curve] (0,5) to node[above,left,pos=0.85]{$\scriptstyle{\zeta}$} (-3,3);
\draw [curve] (0,5)  to node{} (1,1.25);

\draw [curve] (0,5)  to node{} (-0.7,1.25);
\draw [curve] (1,1.25)  to node{} (-0,2.2);

\draw [curve] (-3,3)  to node[above,left,pos=1.3]{$\scriptstyle{e}$} (-0.7,1.25);
\draw [curve] (-0,2.2)  to node{} (-0.7,1.25);
\draw [curve] (0,5)  to node[color=red,pos=1]{$\bullet$} (-0,2.2);

\draw [curve] (1,1.25)  to node{} (-0.7,1.25);
\draw [curve,->] (1.5,3) to node[above]  {\tiny contr.} node[below] {$\scriptstyle{D_c}$} (3,3);
\end{scope}

\begin{scope}[xshift=5.25cm]
\draw [curve] (0,5) to node[above,left,pos=0.85]{$\scriptstyle{\zeta}$} (-3,3);

\draw [curve] (0,5)  to node{} (-0.7,1.25);

\draw [curve] (-3,3)  to node[above,left,pos=1.3]{$\scriptstyle{e}$} (-0.7,1.25);
\draw [curve] (1,1.25)  to node{} (-0.7,1.25);
\draw [curve] (0,5)  to node[color=red,pos=1]{$\bullet$} (1,1.25);

\draw [curve,->] (1.5,3) to node[above]  {\tiny contr.} node[below] {$\scriptstyle{D_{t_0}}$} (3,3);
\end{scope}

\begin{scope}[xshift=12.5cm]
\draw [curve] (0,5) to node[above,left,pos=0.85]{$\scriptstyle{\zeta}$} (-3,3);

\draw [curve] (-0.7,1.25) to node[above,pos=0.95]{$\scriptstyle{s_0}$} (0,5);

\draw [curve] (-3,3)  to node[above,left,pos=1.3]{$\scriptstyle{e}$} (-0.7,1.25);
\end{scope}
\end{tikzpicture}
\end{center}
\caption{A birational map from $\tilde{\mathcal{F}}$ to $\PP(1,3,17,25)$. The origin and $v_{t_1}$ are behind the page.}\label{fig: toric-map}
\end{figure}
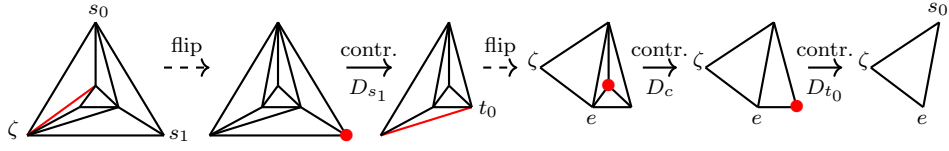

\begin{prop}
A general quasismooth hyper\-surface $S$ of weighted 
degree $51$ in $\PP(1,3,17,25)$ is an RU-surface (here quasismooth means that the affine cone is smooth outside of the vertex).
\end{prop}

\begin{proof}
By adjunction we have $K_S = \Oh_S(51 - 25 - 17 - 3 - 1) = \Oh_S(5)$ which has two global sections and
$K_S^2 = (51\cdot5^2) / (1\cdot3\cdot17\cdot25) = 1$.

Write $e,t_1,s_0,\zeta$ for the coordinates on the weighted projective space then the equation of degree 51
can be expressed as
\begin{equation} \label{eq: RU hypersurface} 
eP_{50} + \tau t_1^{17} + \theta t_1^3s_0\zeta+ s_0^3 
\end{equation}
where $P_{50}(e,t_1,s_0,\zeta)$ is a polynomial of weighted degree 50 and $\theta$, $\tau$ are 
parameters. Note that there are only three forms of degree 51 that are not 
divisible by $e$. It is easy to check by hand that $S$ is quasismooth.

Now $S$ contains the index 5  coordinate point $(0,0,0,1)$ and we get a $\frac1{25}(3,17) \cong \frac1{25}(1,14)$ singularity there because since $P_{50}$ contains the monomial $\zeta^2$.
If $\tau$ is nonzero then $S$ does not contain the singular point of index 3.
We assume that the coefficient of $s_0^3$ is nonzero $(=1)$ to avoid the index 17 
point.

 By \cite[Theorem 3.2 (A1)]{rana-urzua19} the surface constructed is an RU-surface. 
\end{proof}

\begin{rmk}
In coordinates, the map $\tilde{\mathcal{F}}\dashrightarrow\PP(1,3,17,25)$ is given by:
\[(s_1,t_0,c,e,t_1,s_0,\zeta)\mapsto(1,1,1,e,t_1,s_0,\zeta),\]
and this maps the equation \eqref{eqn!Ytilde} defining $\tilde Y$, to the equation \eqref{eq: RU hypersurface} defining $S_{51}$. The degree $51$ can also be read off from the multidegree $(6,18,34,51)$ of $\tilde Y$ with respect to the weight matrix \eqref{eq: weight matrix}. Thus we have a generically 1-1 map between the moduli spaces parametrising 
hypersurfaces $\tilde Y\subset\tilde{ \mathcal{F}}$ and hypersurfaces $S_{51}\subset\PP(1,3,17,25)$.
 In particular, the general element in $\gothM_{RU}$ can be realised as such a hypersurface $S_{51}$.
 
 This agrees with the naive parameter count. In $\PP(1,3,17,25)$, the linear 
system $\PP(H^0(\Oh(51)))$ has dimension $50$ and the automorphism group has 
dimension $22$, which suggests that quasismooth hypersurfaces of weighted  
degree $51$ have $28$ moduli.
\end{rmk}

We now show that the parameters occuring in \eqref{eq: RU hypersurface} play the same role as in Section \ref{sec: canonical-ring}. 
\begin{prop}
The RU-hypersurface with $\tau = 0$ (respectively $\tau=\theta=0$) has an additional $\frac19(1,5)$ (resp.~$\frac1{18}(1,5)$) singularity.
\end{prop}

\begin{proof}
Near the index 3 point $(0,1,0,0)$ the local analytic form of $S$ is
\[(\theta s_0\zeta + e^3 + \text{h.o.t.})=0 \subset \frac13(1_e,2_{s_0},1_{\zeta}),\]
if we assume that the monomial $e^2t_1^{18}$ appears in $P_{50}$. By Remark \ref{rmk!index-1-cover}, this is a $\frac 19(1,5)$ singularity.

Moreover, since $t_1^{11}s_0$, $\zeta^2$, and $et_1^8\zeta$ appear in $P_{50}$, if $\theta=0$ we get $(es_0+e\zeta^2+e^2\zeta+\text{h.o.t.}=0)\subset\frac13(1,2,1)$ which is a $\frac1{18}(1,5)$ singularity. 
\end{proof}

\begin{prop}
The canonical model of an RU-hypersurface is the same as the one described in Corollary \ref{cor!index-5-format}.
\end{prop}

\begin{proof}
As shown above, $K_S=\Oh_S(5)$ so we can write out generators and relations for the canonical ring
directly:
\begin{gather*}
x_0 = e^5,\ x_1 = e^2t_1,\ y = et_1^3,\ w = t_1^5,\ u_0 = e^3s_0,\ u_1 = t_1s_0,\\
z = \zeta,\ t = es_0^2,\ g = s_0^5
\end{gather*}
in degrees $1,1,2,3,4,4,5,7,17$ respectively.

The easy monomial relations between these generators are the same as in Section \ref{sec!relations}
and the equation of degree 51 can be expressed in terms of these new 
generators in five different ways by multiplying it with each of $e^4, et_1, e^3t_1^2, t_1^3, s_0^2$
(cf.~Lemma \ref{lem!rolling}). The last of these five equations involves $s_0^5$ and therefore we can use it to eliminate the spurious 
generator $g$ of degree 17 in the same way as Lemma \ref{lem!rolling}.

Finally, we can fit these relations into the skew-matrix format of Corollary \ref{cor!index-5-format}.
\end{proof}

\begin{rmk}
After an appropriate coordinate change, $P_{50}$ involves only even powers of $\zeta$. Thus the general RU-hypersurface has an involution $(e,t_1,s_0,\zeta)\mapsto(e,t_1,s_0,-\zeta)$ if and only if $\theta$ vanishes.
This is another interpretation of the obstruction to $\QQ$-Gorenstein smoothing of the general RU-surface cf.~\cite[Prop.~3.18]{FPRR}. 
\end{rmk}

\section{Cuspidal RU-surfaces are \texorpdfstring{$\QQ$}{Q}-Gorenstein smoothable}\label{sec: smoothable}

 In this section we assume that $\theta=0$, that is, we consider the cuspidal RU-surfaces. We exhibit a $\QQ$-Gorenstein smoothing of the general cuspidal RU-surface. 
 Since the relations $R_{11},\dots,R_{14}$ are only determined modulo $R_1,\dots,R_{10}$, we use $R_3$ to rewrite $R_{14}$ as \[R_{14}\colon u_0P+\tau y^2w^2u_1+t^2=0.\] 
 The new relation no longer fits into the previous Pfaffian format. This choice of $R_{14}$ is crucial in finding the $\QQ$-Gorenstein smoothing:

\begin{prop}\label{prop!q-gor-index-5}
Consider the family $\mathcal{X}/\Lambda$ defined by relations
\begin{align*}
\tilde R_1\colon x_0y-x_1^3+\lambda^3\tau w &=0 & \tilde R_2\colon x_0w-x_1^2y+\lambda u_0&=0 \\
\tilde R_3\colon x_1w-y^2+\lambda u_1&=0 & \tilde R_4\colon x_0u_1-x_1u_0+\lambda^2\tau yw&=0 \\
\tilde R_5\colon x_1^2u_1-yu_0+\lambda^2\tau w^2&=0 & \tilde R_6\colon x_1yu_1-wu_0-\lambda t&=0 \\
\tilde R_7\colon x_0t-u_0^2 +\lambda\tau x_1y^2w &=0 & \tilde R_8\colon x_1t-u_0u_1+\lambda \tau yw^2&=0\\
\tilde R_9\colon  x_1u_1^2-yt-\lambda\tau w^3&=0 & \tilde R_{10}\colon yu_1^2-wt+\lambda \tilde P&=0
\end{align*}
\begin{align*}
\tilde R_{11}\colon x_0\tilde P+\tau x_1^2w^3+u_0t -\lambda^2\tau wu_1^2&=0\\
\tilde R_{12}\colon x_1\tilde P+\tau yw^3+u_1t&=0\\
\tilde R_{13}\colon y\tilde P+\tau w^4+u_1^3&=0\\
\tilde R_{14}\colon u_0\tilde P+\tau y^2w^2u_1+t^2&=0
\end{align*}
where $\lambda$ is the coordinate on $0\in\Lambda\subset \CC$ and $\tilde P$ is a general polynomial of degree 10 satisfying $\tilde P|_{\lambda=0}=P$. If $\tau\ne0$ then the central fiber $\mathcal{X}_0$ is a cuspidal Rana--Urz{\'u}a surface with a single $\frac1{25}(1,14)$ point, and $\mathcal{X}/\Lambda$ is a $\QQ$-Gorenstein smoothing of $\mathcal{X}_0$. 
\end{prop}

\begin{proof}
By construction, the fiber over $\lambda=0$ is a Rana--Urz\'ua surface because the relations match those of \S\ref{sec!relations}. By Lemma \ref{lem!index-5-format} below, the general fiber with $\lambda\ne0$ is a smooth surface. Hence $\mathcal{X}/\Lambda$ is flat.

The $\QQ$-Gorenstein condition is only relevant near the singular point of $\mathcal{X}$ at $P_z$ over $\lambda=0$. Substituting $z=1$ into the  equations $\tilde R_{10}$, $\tilde R_{11}$, $\tilde R_{12}$, $\tilde R_{13}$, $\tilde R_{14}$ allows us to eliminate $x_0,x_1,y,u_0$ respectively in the same way as the proof of Corollary \ref{cor!index-5-format}, so that we are left with $(u_1^5-wt+\lambda+\text{h.o.t.}=0)$ in $\frac15(1,3,4)\times\Lambda$. Thus the family $\mathcal{X}/\Lambda$ induces a $\QQ$-Gorenstein smoothing of the $\frac1{25}(1,14)$ point on the fiber over $\lambda=0$ (see e.g.~\cite[Ex.~2.1.8]{hacking-barcelona}). 
\end{proof}

\begin{lem}\label{lem!index-5-format}
The family $\mathcal{X}/\Lambda$ fits into the matrix format \[\Pf_4 M_1=\Pf_4 M_2=M_1V_1=M_2V_2=0,\] where
\begin{gather*}
M_1=\begin{pmatrix}
\lambda&x_1 & w & u_1 \\
&u_0 & yu_1 & t\\
&&t&-\tau yw^2 \\
-\text{sym}&&&\tilde P
\end{pmatrix},\ \  M_2=\begin{pmatrix}
\lambda&x_1 & y & t \\
&y & w & u_1^2 \\
&&-u_1&\tau w^3 \\
-\text{sym}&&&-\tilde P
\end{pmatrix}, \\
V_1={}^t\begin{pmatrix}
-\lambda\tau y^2w, & u_0, & -x_1y, & x_0, & 0
\end{pmatrix},\\
V_2={}^t\begin{pmatrix}
u_0, & -\lambda^2\tau w, & x_1^2, & -x_0, & 0
\end{pmatrix}. 
\end{gather*}If $\lambda\ne0$ then the fiber $\mathcal{X}_{\lambda}$ is isomorphic to a nonsingular hypersurface of weighted degree 10 in $\PP(1,1,2,5)$.
\end{lem}

\begin{proof}
Up to sign, the Pfaffians of $M_1$ are $\tilde R_{10}$$, $ $\tilde R_{14}$, $ \tilde R_{12}$, $ \tilde R_6$, $\tilde R_8$ and the Pfaffians of $M_2$ are $\tilde R_{10}$, $ \tilde R_{13}$, $ \tilde R_{12}$, $ \tilde R_3$, $ \tilde R_9$. The product $M_1V_1$ gives $\tilde R_2$, $ y\tilde R_4$, $ \tilde R_7$$, $ $y\tilde R_8$, $\tilde R_{11}+\tau w(x_1w-\lambda u_1)\tilde R_3$ and $M_2V_2$ gives $\tilde R_1$, $\tilde R_2$, $\tilde R_4$, $\tilde R_5$, $\tilde R_{11}$. Thus taken all together these generate the ideal defining $\mathcal{X}/\Lambda$. 

Assume now that $\lambda\ne0$. We perform row and column operations on $M_i$ preserving antisymmetry, and apply the complementary row operations to $V_i$ so that the products $M_iV_i$ are preserved. This gives new matrices $M_i'$ and $V_i'$:
\begin{gather*}
M'_1=\begin{pmatrix}
\lambda&0&0&0 \\
&0&0&0\\
&&\frac{\tilde R_6}{\lambda}&\frac{\tilde R_8}{\lambda} \\
-\text{sym}&&&\frac{\tilde R_{10}}{\lambda}
\end{pmatrix}, \ \ M'_2=\begin{pmatrix}
\lambda&0&0&0 \\
&0&0&0 \\
&&\frac{\tilde R_3}{\lambda}&\frac{\tilde R_9}{\lambda} \\
-\text{sym}&&&\frac{\tilde R_{10}}{\lambda}
\end{pmatrix}\\
V'_1={}^t\begin{pmatrix}
\frac{y\tilde R_4}{\lambda}, & \frac{\tilde R_2}{\lambda}, & -x_1y, & x_0, & 0
\end{pmatrix},\ \
V'_2={}^t\begin{pmatrix}
\frac{\tilde R_2}{\lambda}, & \frac{\tilde R_1}{\lambda}, & x_1^2, & -x_0, & 0
\end{pmatrix}. 
\end{gather*}
Thus the format reduces to $\tilde R_1,\tilde R_2, \tilde R_3, \tilde R_6,\tilde R_8,\tilde R_9,\tilde R_{10},y\tilde R_4$.

Moreover, the assumption $\lambda\ne0$ enables us to rewrite $w,u_0,u_1,t$ using relations $\tilde R_1,\tilde R_2, \tilde R_3, \tilde R_6$ respectively:
\begin{align*}
x_0y-x_1^3&=-\lambda^3\tau w & x_0w-x_1^2y&=-\lambda u_0\\
x_1w-y^2&=-\lambda u_1 & x_1yu_1-wu_0&=\lambda t
\end{align*}
Doing this transforms $\tilde R_8,\tilde R_9,y\tilde R_4$ into identities rather than relations.
For example $\tilde R_4$ reduces to the identity:
\begin{align*}
& \phantom{\equiv\; \; } x_0u_1-x_1u_0+\lambda^2\tau yw\\
&\equiv x_0\left(\frac{x_1w-y^2}{-\lambda}\right)-x_1\left(\frac{x_0w-x_1^2y}{-\lambda}\right)+\lambda^2\tau y\left(\frac{x_0y-x_1^3}{-\lambda^3\tau}\right)\\
&\equiv\frac1\lambda\left(-x_0x_1w+x_0y^2+x_1x_0w-x_1^3y-x_0y^2+x_1^3y\right)\\
&\equiv 0.
\end{align*}
The remaining relation is $\tilde R_{10}$. By repeatedly using $\tilde R_1,\tilde R_2, \tilde R_3, \tilde R_6$ to eliminate $w,u_0,u_1,t$ as above, we are left with a relation between generators $x_0,x_1,y,z$, which we display here, multiplied by $\lambda^{11}$ for readability:
\[x_0(x_0y-x_1^3)^3-3\lambda^3x_1^2y(x_0y-x_1^3)^2+
3\lambda^6\tau^{-1}x_1y^3(x_0y-x_1^3)+\lambda^9y^5+\lambda^{12}\tilde P=0.\]
Since $\tilde P$ is general, this is a nonsingular surface.
\end{proof}
The last equation in the proof can be interpreted as follows:
\begin{cor}
Let $\mathcal X\to \Lambda$ be the $1$-parameter smoothing of a cuspidal RU-surface $X_0$ constructed above. Then there is a diagram
  \[ \begin{tikzcd}
 \mathcal X \arrow{dr} \arrow[dashed]{r}{\phi}& \tilde{\mathcal{X}}\arrow{d}\arrow[hookrightarrow]{r}{\deg 10} & \PP(1,1,2,5)\times \Lambda\arrow{dl}\\
 & \Lambda
    \end{tikzcd}
\]
where $\phi$ is birational and an isomorphism outside the central fibers $X_0$ and $\tilde X_0 =  (x_0(x_1^3-x_0y)^3=0)$. 
\end{cor}
This phenomenon could be taken as a starting point for  a comparison of the closure  of the Gieseker component and a GIT moduli space of hypersurfaces of degree $10$ in $\PP(1,1,2,5)$.

\section{I-surfaces with two T-singularities}\label{sec: two sings}

In this Section, we apply and extend the results of Section  \ref{sec: canonical-ring} to understand how T-divisors in $\overline\gothM_{1, 3}$ intersect each other.
We recall that for  an I-surface $X$ with a unique non-canonical T-singularity $Q$,  the point $Q$ is as in Table \ref{tabular: 1 T-sing} by  \cite[Thm.~1.1]{FPRR}. 
\begin{table}[ht]\caption{T-singularities $\frac{1}{dn^2}(1, dna-1)$ occuring individually} \label{tabular: 1 T-sing}
 \begin{center}
\renewcommand{\arraystretch}{1.2}
   \begin{tabular}{ccll}
   \toprule
Cartier index $n$ & $d$ &   T-singularity& T-string  \\
  \midrule
   $2$ & $d\leq 32$ & $ \frac{1}{4d}(1,2d-1)$  & $[4]$ or $[3,3]$ or $[3, 
2,\dots, 3]$\\
   $3$ & $2$& $ \frac{1}{18}(1,5)$& $[4,3,2]$ \\
   $5$ & $1$ & $\frac{1}{25}(1, 14)$ &$[2,5,3]$\\
  \bottomrule
 \end{tabular}
 \end{center}
 \end{table}
 
Suppose $X$ is an I-surface with two distinct T-singularities  $Q_1$ 
and $Q_2$ of index $i_1$, resp.~$i_2$ (with $i_1\neq i_2$) belonging to the 
above list, where for simplicity of exposition we restrict to the case $d=1$ in the  index $2$ case. 
In particular, these surfaces should correspond to general points in the intersection of  two T-divisors in  $\overline\gothM_{1, 3}$.

We let 
 $f\colon \wt{Y}\to X$ be the minimal desingularization and  
  $\epsilon \colon \wt Y\to Y$ be the morphism to a minimal model: 
 \begin{equation}
\label{diagram}
\begin{tikzcd}
 {} &
  \wt Y \arrow{dl}{f}[swap]{\text{resolution} }  
\arrow{dr}[swap]{\epsilon}{\text{sequence of blow-ups}}   \\
 X& & Y
\end{tikzcd}
\end{equation}
Here $\epsilon$ is  the composition of $k$ blow-ups  at $P_1, 
\dots , P_k$  (possibly infinitely near). We will denote 
the exceptional divisor of $\epsilon$ by $E= \sum_{i=1}^{k} E_i$,
 where the $E_i$'s are the $(-1)$-curves of each blow-up (possibly not 
irreducible, nor reduced). In particular, $K_{\wt{Y}}=\epsilon^*K_Y+E$.

We write 
\[f^*K_X=K_{\tilde Y}+\Delta \]
where $\Delta=\Delta_1+ \Delta_2$ is the codiscrepancy divisor of $f$ with 
$\Delta_j$ supported on $f^{-1}(Q_j)$.  These are chains of smooth rational curves with self-intersections and coefficients as in Table \ref{tab: discrepancies}.

 \begin{table}
  \caption{Codiscrepancy divisors}\label{tab: discrepancies}
  \begin{center} 
\renewcommand{\arraystretch}{1.2}
   \begin{tabular}{cl}
 \toprule 
 T-singularity& $\Delta_j$  \\
  \midrule 
 $ \frac{1}{4}(1,1)$ & $\frac12 A_j $ with $A_j^2=-4$  \\ 
    \hline
   \multirow{2}*{  $ \frac{1}{18}(1,5)$}&  $\frac 23A_j+\frac23B_j+\frac13C_j $  \\ &   with  $ A_j^2=-4,  B_j^2=-3,  C_j^2=-2$ \\ 
     \hline
    \multirow{2}*{$\frac{1}{25}(1, 14)$} &$\frac 35A_j+\frac45B_j+\frac25C_j $  \\ &   with  $  A_j^2=-3,  B_j^2=-5,  C_j^2=-2$\\ 
  \bottomrule 
 \end{tabular}
 \end{center} 
 \end{table}

Since $X$ has only rational singularities then we have  
 $q(Y)=q(\wt Y)=q(X)=0$ and $p_g(Y)=p_g(\wt Y)=p_g(X)=2, $  so the Kodaira dimension of $\wt Y$ is  positive  and the minimal model $Y$ is unique. 
Arguing as in \cite[\S 4]{FPRR} we see that $Y$ is a properly elliptic  surface and 
moreover, by \cite[Prop. 20]{lee98},   we have
 \begin{equation}\label{eq: d-r}K^2_{\wt Y} = 1 + (\Delta_1)^2 +  (\Delta_2)^2 = d_1-r_1 + d_2-r_2 -1
 \end{equation}
where  $r_j$ is the length of the T-string and $d_j$ is as in Table \ref{tabular: 1 T-sing}.

\subsection{Useful results}

 We summarize some results that will be used throughout this section. 
First we recall some well known facts about the structure of the 
$\epsilon$-exceptional divisors.

 \begin{rmk}\label{different Gammas}   
Let $\Gamma$, $\Gamma'$ be two distinct  irreducible $(-1)$-curves on $\wt Y$. Then $\Gamma \cap \Gamma' =\emptyset$ since $\wt Y$ has Kodaira dimension 1. 
 
\end{rmk}

 \begin{rmk}\label{multiplicity}  Let $E= \sum_{i=1}^{k} E_i$ be  the exceptional divisor of $\epsilon$, 
 where the $E_i$'s are the $(-1)$-curves of each blow-up. 
Then every  $E_i$ contains at least one irreducible $(-1)$-curve $\Gamma_i$.   Moreover 
 if $D \subset  E $ is an $\epsilon$-exceptional irreducible  $(-n)$-curve with $n\geq 2$, then
$E-D \geq \sum m_h \Gamma_h$,  with $\Gamma_h$'s irreducible $(-1)$-curves and  $\sum m_h \geq k$.
\end{rmk}

Now we show some   properties  of the components of the  T-strings (i.e. $f$-exceptional divisors) and the $\epsilon$-exceptional divisors.

From now on, we abuse notation and denote the pull back of $K_Y$ to $\wt{Y}$  by $K_Y$, and the pull back of $K_X$ to $\wt{Y}$ by $K_X$. 
With this convention, we can write $$K_X=K_{\wt Y}+ \Delta_1 + \Delta_2 = K_{ Y}+ E  + \Delta_1 + \Delta_2.$$

\begin{rmk}\label{(-1)curves} 
Let $\Gamma \subset  \wt Y$ be an $\epsilon$-exceptional irreducible  $(-1)$-curve. Then $K_X \Gamma >0$. 
\end{rmk}

\begin{lem}\label{lem: -2-curves}Let $D\subset \wt Y $ be an irreducible $(-2)$-curve.

Then $K_Y  D=ED=0$.
In particular:
\begin{enumerate}
\item[\em(i)] if $D\subset \wt Y$ is not contracted by $\epsilon$,  then $\epsilon (D)\subset Y $ is again a $(-2)$-curve and 
$D\cap E = \emptyset$; 

\item[\em(ii)] if $D\subset \wt Y$ is  contracted by $\epsilon$,  then
$D$ intersects only  one $(-1)$-curve $\Gamma \subset E$ and  $D \Gamma =1$.  
\end{enumerate}
\end{lem}
\begin{proof}
Since $D^2=-2$,  by adjunction we have 
\[0 = K_{\wt{Y}} D = K_Y D +ED\]
and $K_Y D\geq 0$ because $K_Y$ is nef.

If $D$ is not $\epsilon$-exceptional then  the second summand is non-negative, giving $ED = 0 $ and the first item. 

If $D$ is $\epsilon$-exceptional then the first summand is zero, so also the second. 
To conclude the proof of   {\em (ii)}  assume that $D$ intersects two distinct $(-1)$-curves $\Gamma, \Gamma'$. Contracting $\Gamma$ and $\Gamma'$, $D$ becomes a curve  with non-negative self intersection 
and negative intersection with the canonical divisor. This is absurd since $Y$ is an elliptic surface. If $D \Gamma \geq 2$ then contracting $\Gamma$ we obtain a curve which is not $\epsilon$-exceptional
and has negative intersection with the canonical divisor, impossible on the minimal surface $Y$.
\end{proof}

\begin{lem}\label{lem: -3-curves} Let $D\subset \wt Y $  be an irreducible 
$(-3)$-curve that is not $\epsilon$-exceptional.
Then $\epsilon (D) \subset Y $ is  either  a $(-3)$-curve or a  $(-2)$-curve. 

\end{lem}
\begin{proof}
Since $D^2=-3$,  by adjunction we have 
\[1 = K_{\wt{Y}} D = K_Y D +ED,\]
  $K_Y D\geq 0$ because $K_Y$ is nef and $ED\ge 0$ since $D$ is not exceptional. If $ED=0$, then $\epsilon(D)$ is again a $-3$-curve. 
So assume $K_YD=0$ and $ED=1$ and write $D=\epsilon^*(\epsilon(D))-\sum m_iE_i$, with $m_i\ge 0$.  We have $1=DE=\sum m_i$, so $D$ is obtained by blowing up $\epsilon(D)$ once at a smooth point and $\epsilon(D)$ is a $-2$-curve.
\end{proof}

\begin{lem}\label{lem: K_XE}
We have 
  $$K_X K_Y = (\Delta_1 +\Delta_2) K_Y  \geq \frac 12, \ \ \ K_X E  \leq \frac12$$
\end{lem}
\begin{proof} We have
$$K_X K_Y = ( K_{ Y}+E + \Delta_1 + \Delta_2) K_Y = ( \Delta_1 + \Delta_2) K_Y$$
since $K_Y^2=K_Y E =0$.  Moreover, we have $K_XK_Y>0$ because $K_Y$ moves and $K_X$ is positive outside the support of the $f$-exceptional locus.
 Looking at  the description of the codiscrepancy divisors  we note that  the  divisors with coefficient less than $\frac12$ are $(-2)$-curves. Then by the above Lemma \ref{lem: -2-curves} we obtain  $( \Delta_1 + \Delta_2) K_Y \geq \frac12 $. 
The second inequality follows since $1=K_X^2 = K_X (K_Y + E )$. 
\end{proof}

\begin{prop}\label{lem: exceptional} 
Let $D \subset  E \subset \wt Y$ be an $\epsilon$-exceptional irreducible  $(-n)$-curve, with $n\geq 2$. Then $D$ is also $f$-exceptional.

\end{prop} 
\begin{proof}
Suppose not, then we have $0< K_X D = K_{\wt{Y}} D + (\Delta_1 +\Delta_2) D= n-2 +  (\Delta_1 +\Delta_2) D$.
If $n\geq 3$ then we get  $K_X D \geq 1$, which is absurd since $K_X D \leq K_X E \leq \frac12$ by Lemma \ref{lem: K_XE}.

Therefore we may assume  $n=2$. 
 We first consider the case where    $Q_1$, $Q_2$ are T-singularities of index $i_1=5$, resp.~$i_2=2$, so that by equation  \eqref{eq: d-r},  $\epsilon \colon \wt{Y}\to Y$ is  the composition of three blow-ups. 
In this case  $K_X D = (\Delta_1 +\Delta_2) D\geq \frac25$  (see the codiscrepancies shown above).
By Remark \ref{multiplicity} and Remark \ref{(-1)curves}, since we have three blow-ups and for every irreducible $(-1)$-curve $\Gamma \subset E$ it is  $K_X \Gamma \geq \frac{1}{10}$, we obtain  $K_X (E-D) \geq \frac{3}{10}$. 
Whence 
$ K_XE =  K_XD + K_X(E-D) \geq \frac25 +\frac{3}{10}>\frac{1}{2} $, which contradicts Lemma \ref{lem: K_XE}. 

The cases $(i_1,i_2)=(5,3)$ and  $(i_1,i_2)=(3,2)$ are similar. 
\end{proof}

 \begin{cor}\label{(-n)-curve}   
Let $D \subset  E \subset \wt Y$ be an $\epsilon$-exceptional irreducible  $(-n)$-curve. 
If $n \geq 2$ then $K_X D =0$; if $n=1$ then $K_X D \geq \frac{1}{i_1 i_2}$. 

\end{cor}

 \subsection{I-surfaces with a singularity of  type $\frac{1}{25}(1, 14)$ and a singularity of index $2$ do not exist}

In this section we are going to prove  the following 
\begin{prop}\label{no 5+2 sing} There are no T-singular I-surfaces with a singularity of type $\frac{1}{25}(1,14)$ and  a singularity of type $\frac{1}{4}(1,1)$. 
\end{prop}
\begin{rmk} A generalisation of the below proof shows that there are no T-singular I-surfaces with a singularity $\frac1{25}(1,14)$ and a singularity of type $\frac1{4d}(1,2d-1)$. 
We do not include the details here but it involves keeping track of the possible intersections with the index 2 T-chain. 
Moreover, we do not know if there is an I-surface with more general singularities of index 5 and index 2. 
\end{rmk}
\begin{proof}

Assume by a contradiction the existence of such a surface and consider the diagram \ref{diagram}
where  the resolution of the two singular points yields a string of type  $[3,5,2]$ and a string of type  $[4]$. 

The strategy of the proof consists in studying the possible configuration of $\epsilon$-exceptional irreducible  $(-1)$-curves. 

\vspace{ 5 mm} 

First note that by equation  \eqref{eq: d-r},   $\epsilon\colon \wt{Y} \to Y $ is a composition of three blow-ups. 
Let $Q_1$ be the point of index 5 and $Q_2$ be the point of index 2. The codiscrepancy divisor corresponding to $Q_1$ (respectively $Q_2$) is $\Delta_1=  \frac 35A_1+\frac45B_1+\frac25C_1$ (resp.~$\Delta_2 = \frac12 A_2$). Thus we can write
$$K_X = K_Y + \sum_{i=1}^3 E_i + \Delta_1 + \Delta_2.$$

By Remark \ref{multiplicity}, Lemma \ref{lem: K_XE}, Corollary \ref{(-n)-curve} we have  $\frac3{10}\le K_X E \le\frac12$ and $\frac12\le K_X K_{Y} \le \frac7{10}$.
Now $$K_X K_{ Y}=  (\Delta_1 +\Delta_2) K_Y=  (\frac 35A_1+\frac45B_1+\frac25C_1 + \frac12  A_2)K_Y,$$ hence the only possibilities for  $K_X K_Y$ are $\frac12, \frac35$ and for $K_X E$ we get:
 \begin{equation}\label{EK_X} K_X E=\frac12  \ \ \mbox{ or }   \ \ K_X E=\frac25 . 
 \end{equation}

 Now let $\Gamma$ be an   $\epsilon$-exceptional irreducible $(-1)$-curve.  We have $ K_X \Gamma \geq \frac{1}{10}$ since $X$ has index 10. 
Now, since $K_X (E - \Gamma) \geq \frac{2}{10}$ we obtain  $ K_X \Gamma \leq \frac{3}{10}$. Hence
since $K_X = K_{\wt{Y}} + \Delta_1 + \Delta_2$ and $K_{\wt Y}\Gamma=-1$,   we get 
$\frac{11}{10} \leq ( \Delta_1 + \Delta_2)\Gamma \leq \frac{13}{10}$.

 We exclude the cases where $\Gamma C_1 =2, 3$ and  $\Gamma A_1 =2$ since  they contradict Lemma \ref{lem: -2-curves}(i) and  Lemma \ref{lem: -3-curves}.
Therefore we are left with 
 the following possibilities: 
$$\begin{array}{lllllr}
( \Delta_1 + \Delta_2)\Gamma = \frac{3}{5} +\frac12 &=& \frac 35A_1 \Gamma + \frac12A_2 \Gamma & ;  & K_X\Gamma =\frac{1}{10} &   (I) \\ 
& & & & & \\ 
( \Delta_1 + \Delta_2)\Gamma = \frac{4}{5}  +\frac25  &=& \frac 45B_1 \Gamma + \frac25 C_1 \Gamma &  ;  & K_X\Gamma =\frac{2}{10} &  \ \ \ \  (II) \\ 
& & & & & \\ 
( \Delta_1 + \Delta_2)\Gamma = \frac{4}{5}  +\frac12  &=& \frac 45B_1 \Gamma + \frac12A_2 \Gamma &  ;  & K_X\Gamma =\frac{3}{10} &  (III)  \\ 
\end{array}
$$

\begin{figure}
\begin{center}
\begin{tikzpicture}
[curve/.style ={thick, every loop/.style={looseness=10, min distance=30}},
exceptional/.style = {Orange},
scale = 0.6
]

\begin{scope}[xshift=-17.5cm]

\draw[thick, rounded corners] (-3, -2) rectangle (3.4, 4) node[ right]{};
\draw [curve, orange] (-1,-1) node[left]{} to  (-1,3) node[above] {$\Gamma$};

\draw [curve] (-1.5,-0.5) node[below, right]{ \small \ \ $A_2$} to[  looseness = 1]   ++(2.5,0);

\draw [curve,Plum] (-1.8,2) node[above left ]{ } to[  looseness = 1]   ++(1.5,0);
\draw [curve,Plum] (-0.8,2) node[right  ]{} to[  looseness = 1]   ++(2.5,0);
\draw [curve,Plum] (1.3,2) node[above right  ]{\small \  $\Delta_1$} to[  looseness = 1]   ++(1.5,0);

\end{scope}

\begin{scope}[xshift=-10.5cm]

\draw[thick, rounded corners] (-3, -2) rectangle (3.4, 4) node[ right]{};
\draw [curve, orange] (1.2,-1) node[left]{} to  (1.2,3) node[above] {$\Gamma$};

\draw [curve] (-0.5,-0.5) node[below, right]{\small $A_2$ } to[  looseness = 1]   ++(2.5,0);

\draw [curve,Plum] (-1.8,2) node[above left ]{ } to[  looseness = 1]   ++(1.5,0);
\draw [curve,Plum] (-0.8,2) node[right  ]{} to[  looseness = 1]   ++(2.5,0);
\draw [curve,Plum] (1.3,2) node[above right  ]{\small \  $\Delta_1$} to[  looseness = 1]   ++(1.5,0);

\end{scope}

\begin{scope}[xshift=-3.5cm]

\draw[thick, rounded corners] (-3, -2) rectangle (3.4, 4) node[ right]{};

\draw [curve] (-0.5,-0.5) node[below, right]{\small $A_2$ } to[  looseness = 1]   ++(2.5,0);

\draw [curve,Plum] (-1.8,2) node[above left ]{ } to[  looseness = 1]   ++(1.5,0);
\draw [curve,Plum] (-0.8,2) node[right  ]{} to[  looseness = 1]   ++(2.5,0);
\draw [curve,Plum] (1.3,2) node[above right  ]{\small \ \ \  $\Delta_1$} to[  looseness = 1]   ++(1.5,0);

\draw [curve,orange] (-0.2,2) node[below right] {$\Gamma$} to[  looseness = 2]   ++(2.6,0);

\end{scope}

\end{tikzpicture}
\end{center}
\caption{Configurations of type ({\em I} ), ({\em III} ) and ({\em II} )}
\end{figure}
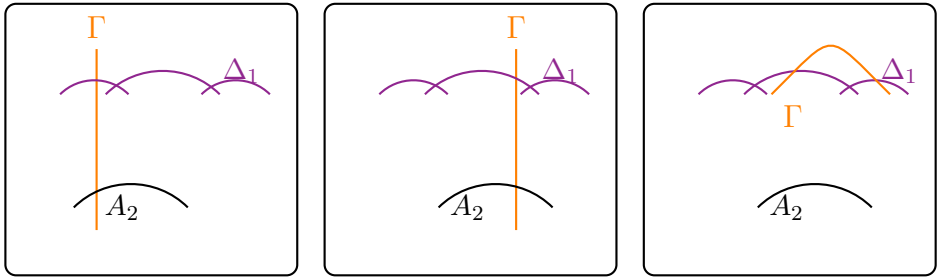

We will show that any of the above configurations gives a contradiction. 
Note that, since  in view of Prop.~\ref{lem:  exceptional} none of the $\Gamma$'s gives a three step contraction on its own, we need to combine them.

\hfill\break 
{\em Configuration (II).}
Assume that there is a curve $\Gamma$ of type ({\em II} ).
Since $K_X (\sum_{i=1}^{3} E_i)= \frac12 $ or $\frac25$, there is a second exceptional curve $\Gamma'$ of type ({\em I} ).

First blow down $\Gamma$. Then the image of $C_1$   is a $(-1)$-curve which intersects  the image of $B_1$ in two points.  Blowing this down
we obtain  a nodal curve  which intersects the image of $A_1$  in 1 point.

 Now  blow down  $\Gamma'$. 
 We obtain a minimal elliptic  surface $Y$  with   $\epsilon(A_1)K_Y  =\epsilon(B_1)K_Y =0$, 
which implies that $\epsilon(A_1)+  \epsilon(B_1)$   is contained in a fiber, contradicting Kodaira's list.

So a type ({\em II} ) configuration does not exist. 

\hfill\break 
{\em Configuration} ({\em III} ).
Consider  a type ({\em III} ) curve $\Gamma$. Since  a type ({\em II}) configuration does not exist, then by equation \eqref{EK_X}  there is a second exceptional curve $\Gamma'$ of type ({\em I}).

Then blowing down $\Gamma$ and $\Gamma'$  we see that  there are no $(-1)$-curves 
arising from $\Delta_1 + \Delta_2$. Hence there exists a third  curve $\Gamma''$ of type (I). Blowing it down, we obtain  another  $(-1)$-curve, which is absurd. 

So a type ({\em III} ) configuration does not exist.

\hfill\break 
{\em Configuration} ({\em I} ). 
 We are left with the case where all the exceptional curves are of type ({\em I} ).  If there exist two curves $\Gamma, \Gamma '$ of type ({\em I} ),  then 
 blowing them down and arguing as in the previous case we obtain a curve having  negative intersection with the canonical divisor,
  which is absurd. Since, as we noticed earlier, there are at least two irreducible $-1$-curves on $\wt{Y}$, the proof is complete.
    \end{proof}

 \subsection{The divisor of surfaces with one singularity of index 3}\label{sec: divisor}

In this subsection we study the divisor of surfaces with an additional singularity of index $3$, and show that it fits into the original Pfaffian format of \S\ref{sec!relations}.

\begin{lem}\label{lem!index-15}
In the notation of \S\ref{sec!relations}, 
if $\tau=0$ and $P_{10}$, $\theta$ are general then $X$ has a $\frac 1{9}(1,5)$ singularity in addition to the $\frac1{25}(1,14)$ singularity. 
\end{lem}
\begin{proof}
The index 3 coordinate point $P_w$ is contained in $X$ because $w^4$ no longer appears in $R_{13}$ when $\tau=0$. In a neighbourhood of $P_w$, the relations $R_2,R_3,R_6,R_{10}$ eliminate $x_0,x_1,u_0,t$ respectively. Thus the local coordinates at $P_w$ are $y,u_1,z$. Since $\theta\ne0$ in general, $R_{13}$ locally defines $X$ as $u_1z=y^3+\dots$ in $\frac13(1_{u_1},2_z,2_y)$ (because in general the monomial $y^2w^2$ appears in $P_{10}$). By Remark \ref{rmk!index-1-cover}, this is a $\frac19(1,5)$-singularity.\end{proof}

 \begin{prop}\label{prop: 3+5} Suppose that $\tau=0$ and consider the surfaces $X_{\lambda,\theta}$ in $\PP(1,1,2,3,4,4,5,7)$ defined by
\[\Pf_4 \tilde M=\tilde M\tilde V=0\]
where
\[
\tilde M=\begin{pmatrix}
0&0&x_0 & x_1^2 & u_0 \\
&0&x_1 & y & u_1 \\
&&u_0 & x_1u_1 & t\\
&&&t&\theta u_1z\\
&&&&P\\
-\text{sym}&&&&
\end{pmatrix},\
\tilde V=\begin{pmatrix}
0\\-u_1^2\\ 0 \\w\\-y\\\lambda\\
\end{pmatrix} 
\]
and $\lambda,\theta$ are parameters satisfying $\lambda\theta=0$.

If $\lambda=\theta=0$ then $X_{0,0}$ is a surface with one $\frac1{25}(1,14)$-singularity and one $\frac1{18}(1,5)$-singularity.

If $\lambda\ne0$, $\theta=0$ then $X_{\lambda,0}$ is a surface $X_{3,10}\subset\PP(1,1,2,3,5)$ with one 
$\frac1{18}(1,5)$ singularity.

If $\lambda=0$, $\theta\ne0$ then $X_{0,\theta}$ is an RU-surface with an extra $\frac1{9}(1,5)$ singularity as in Lemma \ref{lem!index-15}.
\end{prop}
\begin{proof}Clearly, if $\theta\ne0$ then we are in the situation described by Lemma \ref{lem!index-15}.

If $\lambda\ne0$, we can adapt the proof of Proposition \ref{prop!q-gor-index-5}. This time, since $\tau=0$, the generator $w$ can not be eliminated. It turns out that the general fiber $X_{\lambda,0}$ has equations
\[X_{\lambda,0}\colon(x_0y-x_1^3=\lambda^3P+y^5-3x_1y^3w+3x_1^2yw^2-x_0w^3=0)\subset\PP(1,1,2,3,5).\]
That is, $X_{\lambda,0}$ has a $\frac1{18}(1,5)$ singularity (see \cite{FPRR}).

If $\theta=\lambda=0$ then we get a RU-surface with an extra singularity of index 3. Near the coordinate point $P_w$ we use $R_2,R_3,R_6,R_{10}$ to eliminate $x_0,x_1,u_0,t$ respectively. Then $R_{13}$ cuts out a hypersurface 
\[(yu_1+yz^2+y^2z+\dots=0)\subset\frac13(1_{u_1},2_y,2_z),\]
because the monomial $w^2u_1$ appears in $P_{10}$.
This is local analytically a $\frac1{18}(1,5)$-singularity.
\end{proof}

We have thus shown that both the Gieseker component and the RU-component contain in their closures a divisor parametrising surfaces with an additional T-singularity of index $3$ and these divisors meet the intersection divisor, i.e., the divisor of cuspidal RU-surfaces in an irreducible subset of codimension two parametrising surfaces with one $\frac1{25}(1,14)$-singularity and one $\frac1{18}(1,5)$-singularity.

These correspond to the central fiber of the family ($\lambda = \tau = \theta = 0 $) and their minimal resolution is an elliptic surface with a $(-3)$-section and a singular fiber of type III (see Lemma \ref{lem: elliptic surface}). Thus geometrically, these surfaces can be obtained as follows.

\begin{ex}\label{ex: sing 5+3}
 We consider an elliptic surface with $p_g=2$,  
 a fiber of type III,  and a $(-3)$-section.

We blow-up the singular point $p_1$ and its infinitesimal point $p_2$ given by the intersection of the two branches of the singular fiber.  We blow-up  two more points as shown  in  Figure  \ref{fig: 5+3}.

With this procedure we obtain   
a string $[2,5,3] $ and a string $[4,3,2]$ connected by a $(-1)$-curve. Blowing down the two strings we obtain an I-surface with a singularity of type $\frac{1}{18}(1,5)$ and a singularity of type $\frac{1}{25}(1,14)$. 
\end{ex}

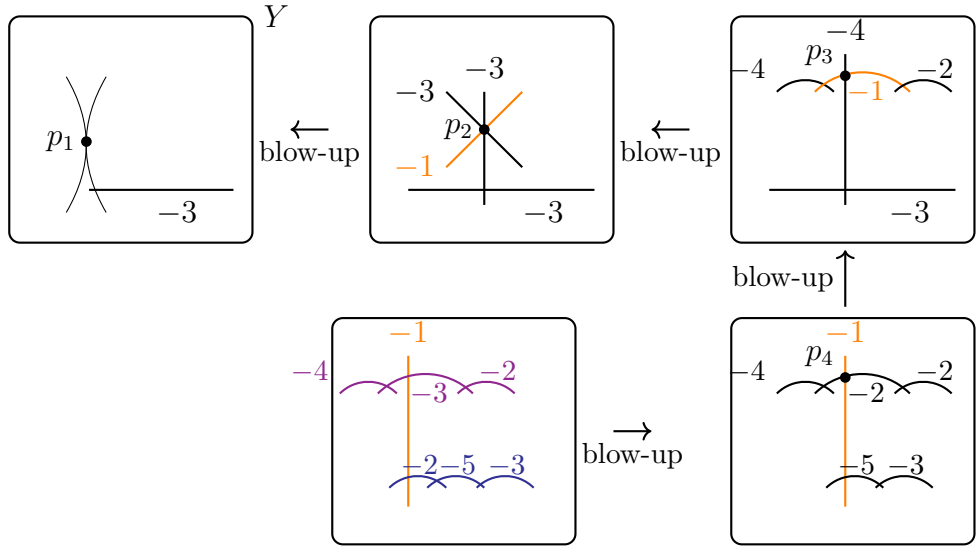
\begin{figure}[ht]
\begin{center}
\begin{tikzpicture}
[curve/.style ={thick, every loop/.style={looseness=10, min distance=30}},
exceptional/.style = {Blue},
scale = 0.5
]

\begin{scope}[xshift=-17cm]
\draw[thick, rounded corners] (-3, -2) rectangle (3.4, 4) node[ right]{$Y$};
\draw (-1.5, -1.2) to[bend right] node[pos = 0.8, right]{} ++(0,3.6);
\draw (-0.45, -1.2) to[bend left] node[pos = 0.8,left] {} ++(0,3.6);

\draw [curve] (-0.9,-0.6)  to node [right,below] { \ \ \ $-3$}  node[above] {} (2.9,-0.6);
\draw [curve,<-] (4.4, 1) to node[above]  {} node[below] {\small blow-up} ++(1,0);
\coordinate (p1) at (-0.97, 0.68);
\fill (p1) circle [radius = 4pt] node[left] {$p_1$};
\end{scope}

\begin{scope}[xshift=-7.5cm]
\draw[thick, rounded corners] (-3, -2) rectangle (3.4, 4) node[ right]{};
\draw [curve] (-1,2) node[below, left]{$-3$} to  ++ ( 2,-2 );
\draw [curve, orange] (-1 , 0) node[above,left ]{$-1$} to  ++(2,2) node[ below ]{};
\draw [curve] (0,-1) node[left]{} to  (0,2) node[above] {$-3$};
\draw [curve] (-2,-0.6)  to node [right,below] { \ \ \ \ \ \ \ \ $-3$}  node[above] {} (2.9,-0.6);
\draw [curve,<-] (4.4, 1) to node[above]  {} node[below] {\small blow-up} ++(1,0);
\coordinate (p2) at (0, 1);
\fill (p2) circle [radius = 4pt] node[left] {$p_2$};
\end{scope}

\begin{scope} [xshift=2cm]
\draw[thick, rounded corners] (-3, -2) rectangle (3.4, 4) node[ right]{};
\draw [curve] (0,-1) node[left]{} to  (0,3) node[above] {$-4$};
\draw [curve] (-2,-0.6)  to node [right,below] { \ \ \ \ \ \ \ \ \ $-3$}  node[above] {} (2.9,-0.6);

\draw [curve] (-1.8,2) node[above left ]{ \small $-4$} to[  looseness = 1]   ++(1.5,0);
\draw [curve, orange] (-0.8,2) node[right  ]{ \ \   \small $-1$} to[  looseness = 1]   ++(2.5,0);
\draw [curve] (1.3,2) node[above right  ]{\small \  $-2$} to[  looseness = 1]   ++(1.5,0);

\draw [curve,<-] (0, -2.2 ) to  node[left]{\text{\small blow-up}} node[right]{} ++(0,-1.5);
\coordinate (p3) at (0, 2.425);
\fill (p3) circle [radius = 4pt] node[above left] {$p_3$};

\end{scope}

\begin{scope}[yshift = -8cm, xshift=2cm]

\draw[thick, rounded corners] (-3, -2) rectangle (3.4, 4) node[ right]{};
\draw [curve, orange] (0,-1) node[left]{} to  (0,3) node[above] {$-1$};
\draw [curve] (-0.5,-0.5) node[above right ]{ \small $-5$} to[  looseness = 1]   ++(1.5,0);
\draw [curve] (0.8,-0.5) node[above  right]{ \small $-3$} to[  looseness = 1]   ++(1.5,0);
\coordinate (p4) at (0, 2.425);
\fill (p4) circle [radius = 4pt] node[above left] {$p_4$};

\draw [curve] (-1.8,2) node[above left ]{ \ \ \small $-4$} to[  looseness = 1]   ++(1.5,0);
\draw [curve] (-0.8,2) node[right  ]{ \ \  \small $-2$} to[  looseness = 1]   ++(2.5,0);
\draw [curve] (1.3,2) node[above right  ]{\small \  $-2$} to[  looseness = 1]   ++(1.5,0);

\end{scope}

\begin{scope}[yshift = -8cm, xshift=-8.5cm]

\draw[thick, rounded corners] (-3, -2) rectangle (3.4, 4) node[ right]{};
\draw [curve, orange] (-1,-1) node[left]{} to  (-1,3) node[above] {$-1$};
\draw [curve,exceptional] (-1.5,-0.5) node[above right ]{ \small $-2$} to[  looseness = 1]   ++(1.5,0);

\draw [curve,exceptional] (-0.5,-0.5) node[above right ]{ \small $-5$} to[  looseness = 1]   ++(1.5,0);
\draw [curve,exceptional] (0.8,-0.5) node[above  right]{ \small $-3$} to[  looseness = 1]   ++(1.5,0);

\draw [curve,Plum] (-2.8,2) node[above left ]{ \small $-4$} to[  looseness = 1]   ++(1.5,0);
\draw [curve,Plum] (-1.8,2) node[right  ]{ \ \   \small $-3$} to[  looseness = 1]   ++(2.5,0);
\draw [curve,Plum] (0.3,2) node[above right  ]{\small \  $-2$} to[  looseness = 1]   ++(1.5,0);

\draw [curve,->] (4.4, 1) to node[above]  {} node[below] {\small blow-up} ++(1,0);

\end{scope}

\end{tikzpicture}
\end{center}
\caption{Construction of an  I-surface with a singularity of type $\frac{1}{18}(1,5)$ and a singularity of type $\frac{1}{25}(1,14)$.}\label{fig: 5+3}

\end{figure}

There is a second way to construct an I-surface with a singularity of type $\frac1{18}(1,5)$ and a singularity of type $\frac1{25}(1,14)$.  Algebraically, we assume that $\lambda = \tau=0$ and the coefficient of $t_1^{16}$ in $l_{16}'(t_0,t_1)$ vanishes (again, see Lemma \ref{lem: elliptic surface}). Then the elliptic surface $Y$ has an $I_3$ fiber over $(0,1;0,1,0)$. Since $\theta$ is generic here, this example is not smoothable as can also be deduced from the non-existence of an involution.

\begin{ex}\label{ex2: sing 5+3}
 We consider an elliptic surface with $p_g=2$,  
 a fiber of type $I_3$,  and a $(-3)$-section.

We blow-up the singular points $p_1,p_2$  of the singular fiber, and two infinitesimal point $p_3,p_4$  as shown  in  Figure  \ref{fig2: 5+3}.

With this procedure we obtain   
a string $[2,5,3] $ and a string $[4,3,2]$ connected by two  $(-1)$-curves. Blowing down the two strings we obtain an I-surface with a singularity of type $\frac{1}{18}(1,5)$ and a singularity of type $\frac{1}{25}(1,14)$. 
\end{ex}

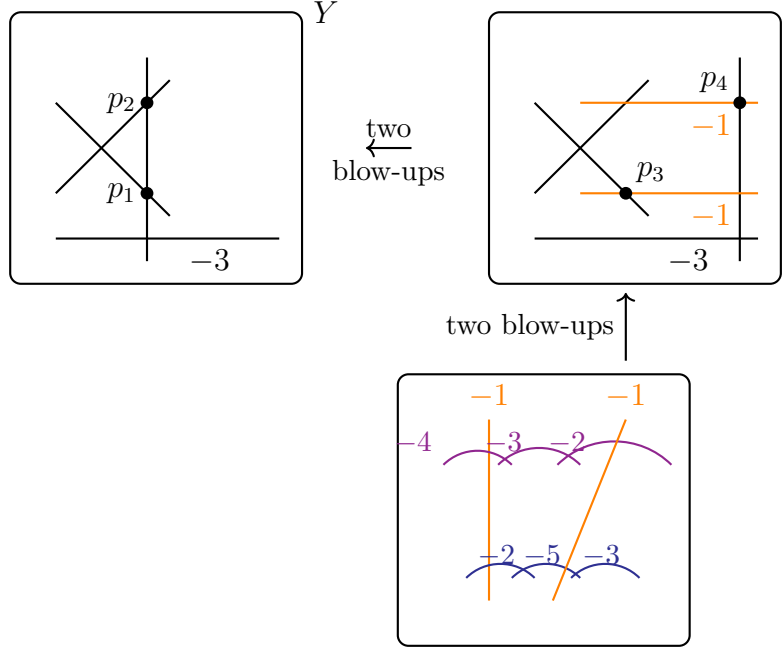
\begin{figure}
\begin{center}
\begin{tikzpicture}
[curve/.style ={thick, every loop/.style={looseness=10, min distance=30}},
exceptional/.style = {Blue},
scale = 0.6
]

\begin{scope}[xshift=-17cm]
\draw[thick, rounded corners] (-3, -2) rectangle (3.4, 4) node[ right]{$Y$};
\draw [curve] (-2,2) node[below, left]{} to  ++ ( 2.5,-2.5);
\draw [curve] (-2 , 0) node[above,left ]{} to  ++(2.5,2.5) node[ below ]{};
\draw [curve] (0,-1.5) node[left]{} to  (0,3) node[above] {};
\draw [curve] (-2,-1)  to node [right,below] { \ \ \ \ \ \ \ \ $-3$}  node[above] {} (2.9,-1);
\draw [curve,<-] (4.8, 1) to node[above]  {\small two} node[below] {\small blow-ups} ++(1,0);

\coordinate (p1) at (0, 0);
\fill (p1) circle [radius = 4pt] node[left] {$p_1$};

\coordinate (p2) at (0, 2);
\fill (p2) circle [radius = 4pt] node[left] {$p_2$};

\end{scope}

\begin{scope}[xshift=-6.5cm]
\draw[thick, rounded corners] (-3, -2) rectangle (3.4, 4) node[ right]{};
\draw [curve] (-2,2) node[below, left]{} to  ++ ( 2.5,-2.5);
\draw [curve] (-2 , 0) node[above,left ]{} to  ++(2.5,2.5) node[ below ]{};
\draw [curve] (2.5,-1.5) node[left]{} to  (2.5,3) node[above] {};
\draw [curve] (-2,-1)  to node [right,below] { \ \ \ \ \ \ \ \ $-3$}  node[above] {} (2.9,-1);

\draw [curve,orange] (-1,0)  to node [right,below] { \ \ \ \ \ \ \ \ $-1$}  node[above] {} (2.9,0);

\draw [curve,orange] (-1,2)  to node [right,below] { \ \ \ \ \ \ \ \ $-1$}  node[above] {} (2.9,2);
\draw [curve,<-] (0, -2.2 ) to  node[left]{\text{\small two blow-ups}} node[right]{} ++(0,-1.5);

\coordinate (p3) at (0, 0);
\fill (p3) circle [radius = 4pt] node[above right] {$p_3$};

\coordinate (p4) at (2.5, 2);
\fill (p4) circle [radius = 4pt] node[above left] {$p_4$};

\end{scope}

\begin{scope}[yshift = -8cm, xshift=-8.5cm]

\draw[thick, rounded corners] (-3, -2) rectangle (3.4, 4) node[ right]{};
\draw [curve, orange] (-1,-1) node[left]{} to  (-1,3) node[above] {$-1$};
\draw [curve,exceptional] (-1.5,-0.5) node[above right ]{ \small $-2$} to[  looseness = 1]   ++(1.5,0);

\draw [curve,exceptional] (-0.5,-0.5) node[above right ]{ \small $-5$} to[  looseness = 1]   ++(1.5,0);
\draw [curve,exceptional] (0.8,-0.5) node[above  right]{ \small $-3$} to[  looseness = 1]   ++(1.5,0);

\draw [curve,Plum] (-2,2) node[above left ]{ \small $-4$} to[  looseness = 1]   ++(1.5,0);
\draw [curve,Plum] (-0.8,2) node[above]{\  \small $-3$} to[  looseness = 1]   ++(1.8,0);
\draw [curve,Plum] (0.5,2) node[above]{\small \ \  $-2$} to[  looseness = 1]   ++(2.5,0);

\draw [curve, orange] (0.4,-1) node[left]{} to  (2,3) node[above] {$-1$};

\end{scope}

\end{tikzpicture}
\end{center}
\caption{Construction of an  I-surface with a singularity of type $\frac{1}{18}(1,5)$ and a singularity of type $\frac{1}{25}(1,14)$.}
\label{fig2: 5+3}

\end{figure}

\begin{rmk} Let $X$ be an I-surface with a singularity of type $\frac{1}{18}(1,5)$ and a singularity of type $\frac{1}{25}(1,14)$.
Arguing as in Proposition \ref{no 5+2 sing} one can see that 
 $X$ is as in the above Examples \ref{ex: sing 5+3}, \ref{ex2: sing 5+3}.
\end{rmk}

\subsection{I-surfaces with a singularity of  index 2 and a singularity of index 3}\label{subsection: sing 2+3}
Let us write the results of \cite{FPRR} in a slightly different form to exhibit clearly the intersection of the  index $2$ and index $3$ divisors: consider the surfaces
\[X=X_{\mu,\nu,\tilde f}\colon\begin{pmatrix}x_0y-x_1^3-\mu u=0 \\ z^2-\nu y^5 - \tilde f_{10}(x_0,x_1,y,u)=0\end{pmatrix}
\subset\PP(1,1,2,3,5),\]
where $\mu, \nu$ are parameters and $\tilde f_{10}$ is sufficiently general  but not containing the monomial $y^5$. 
This is an admissible family of stable surfaces and
\begin{itemize}
 \item for $\mu\nu\neq0$ we can eliminate the variable $u$ via the first equation and get a classical I-surface;
 \item for $\mu\neq 0$, $\nu=0$, we again eliminate $u$ but the branch divisor passes through the vertex of the cone and we acquire an $\frac 14 (1,1)$ point;
 \item for $\mu= 0$, $\nu\neq 0$, we get a general member in the divisor parametrising I-surfaces with an $\frac{1}{18}(1,5)$ point;
 \item for $\mu = \nu = 0$ we get an I-surface with both an $\frac 14 (1,1)$ point and an $\frac{1}{18}(1,5)$ point.  Indeed, at $P_y$, the second equation reduces to $z^2-u^2+x_1^2+\text{h.o.t.}=0$ in $\frac12(1_{x_1},1_{u},1_z)$ after eliminating $x_0$ using $x_1^3$ via the first equation. This is a $\frac14(1,1)$ singularity by Remark \ref{rmk!index-1-cover}.
\end{itemize}
A small dimension count thus confirms that these two T-divisors in the closure of the Gieseker components intersect as expected in an irreducible subset of codimension two whose general element is a surface as in the last item. 

We  complement this algebraic discussion by an  explicit geometric construction of surfaces in the intersection. 

\begin{ex}\label{ex2: sing 2+3}
 We consider an elliptic surface with $p_g=2$, an $I_2$ fiber, an $I_r$ fiber ($r\ge2$) and a $(-3)$-section.

We blow-up the singular points $p_1,p_2$ of the fiber of type $I_2$.

With this procedure we obtain   
a string $[4] $ and a string $[4,3,2]$ connected by two $(-1)$-curves (see Figure \ref{fig2: 2+3}). Blowing down the two strings we obtain an I-surface with a singularity of type $\frac{1}{18}(1,5)$ and a singularity of type $\frac{1}{4}(1,1)$. 
\end{ex}

\begin{figure}
\begin{center}
\begin{tikzpicture}
[curve/.style ={thick, every loop/.style={looseness=10, min distance=30}},
exceptional/.style = {Orange},
scale = 0.7
]

\begin{scope}[xshift=-9.5cm]
\draw[thick, rounded corners] (-2.5, -1.5) rectangle (2.5, 3) node [above right] {$Y$};
\begin{scope}[curve]
\draw  (-2,-0.3)  to  node[above]{\small $-3$} (2,-0.3);
\draw[line width = .2cm,   white] (1.2, -1) to[bend right]  ++(0,3.6);
\draw (1.2, -1.2) to[bend right] node[pos = 0.8, right]{\small $-2$} ++(0,3.6);
\draw (1.5, -1.2) to[bend left] node[pos = 0.8,left] {\small $-2$} ++(0,3.6);
\draw (-1.4, -1.2) to node[pos = 0.8, right]{\small $-2$} ++(0,3.6);
\coordinate (p1) at (1.35, -0.9);
\fill (p1) circle [radius = 4pt] node[right] {\small $p_1$};
\coordinate (p2) at (1.35, 2.1);
\fill (p2) circle [radius = 4pt] node[right] {\small $p_2$};
\end{scope}
\draw [curve,<-] (3.7, 1) to node[above] {two} node[below] {\small blow-ups} ++(2,0);
\end{scope}

\begin{scope}
\draw[thick, rounded corners] (-2.5, -1.5) rectangle (2.5, 3) node [above right] {$\tilde Y$};

\draw [curve, orange] (1.1,-1) node[left]{} to  (1.1,2.3) node[above] {\small $-1$};
\draw [curve, orange] (2,-1) node[left]{} to  (2,2.3) node[above] {\small $-1$};

\draw [curve,blue] (0,1.4) node[above]{ \small $-4$} to[  looseness = 1]   ++(2.2,0);

\draw [curve,Plum] (-1.4,-1) to node[pos=0.8,left]{\small $-2$} ++(0,3.4);
\draw [curve,Plum] (-1.6,-0.5) node[right]{\ \ \small $-3$} to[  looseness = 1]   ++(2,0);
\draw [curve,Plum] (0,-0.5) node[right]{\small  $-4$} to[  looseness = 1]   ++(2.2,0);
\end{scope}

\end{tikzpicture}
\end{center}
\caption{Construction of an  I-surface with a singularity of type $\frac{1}{18}(1,5)$ and a singularity of type $\frac{1}{4}(1,1)$.}\label{fig2: 2+3}

\end{figure}
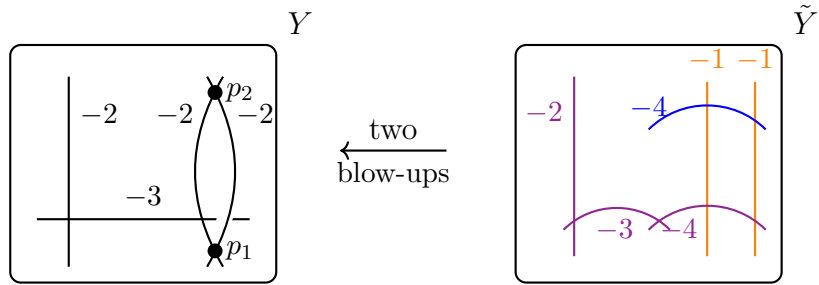

\begin{rmk} Let $X$ be an I-surface with a singularity of type $\frac{1}{18}(1,5)$ and a singularity of type $\frac{1}{4}(1,1)$.
Arguing as in Proposition \ref{no 5+2 sing} one can see that 
 $X$ is as in the above Example \ref{ex2: sing 2+3}.
\end{rmk}

\bibliographystyle{alpha}
\bibliography{cfprr}

\end{document}